\documentclass[reqno, 12pt]{amsart}
\usepackage[a4paper, margin=1in]{geometry}
\usepackage{amsfonts, amssymb, amsmath, amsthm, mathrsfs, enumerate, enumitem, setspace, url}
\usepackage{pifont}
\usepackage[utf8]{inputenc}
\usepackage{graphicx}
\usepackage{tikz-cd} %diagrams
\usetikzlibrary{cd}
\usetikzlibrary{graphs}
\usepackage[all]{xy}
\usepackage{mathtools}
\usepackage{color}
\usepackage{caption}
\usepackage{hyperref} %hyperlinks
\definecolor{darkred}{rgb}{0.5,0,0}
\definecolor{darkgreen}{rgb}{0,0.5,0}
\definecolor{darkblue}{rgb}{0,0,0.5}
\hypersetup{linkcolor=darkblue,filecolor=darkgreen,urlcolor=darkred,citecolor=darkblue}

\usepackage[normalem]{ulem}
\usepackage[toc,page]{appendix}

\theoremstyle{plain}
\newtheorem{theorem}{Theorem}[section]
\newtheorem*{theorem*}{Theorem}
\newtheorem{proposition}[theorem]{Proposition}
\newtheorem{lemma}[theorem]{Lemma}
\newtheorem{corollary}[theorem]{Corollary}

\theoremstyle{remark}
\newtheorem{remark}[theorem]{Remark}

\newtheorem*{acknowledgements}{Acknowledgements}

\theoremstyle{definition}
\newtheorem{definition}[theorem]{Definition}

\numberwithin{equation}{section}

\renewcommand{\(}{\left(}
\renewcommand{\)}{\right)}

\newcommand{\ZZ}{\mathbb{Z}}
\newcommand{\QQ}{\mathbb{Q}}
\newcommand{\Zp}{\mathbb{Z}_p}

\newcommand{\Qp}{\mathbb{Q}_p}

\newcommand{\RR}{\mathbb{R}}

\newcommand{\Fp}{\mathbb{F}_p}

\renewcommand{\AA}{\mathbb{A}}
\newcommand{\PP}{\mathbb{P}}

\newcommand{\sA}{\mathcal{A}}
\newcommand{\sB}{\mathcal{B}}

\newcommand{\sD}{\mathcal{D}}

\newcommand{\sO}{\mathcal{O}}

\newcommand{\sU}{\mathcal{U}}
\newcommand{\sV}{\mathcal{V}}
\newcommand{\sX}{\mathcal{X}}

\newcommand{\frakp}{\mathfrak{p}}

\DeclareMathOperator{\HH}{H}
\DeclareMathOperator{\Gal}{Gal}
\DeclareMathOperator{\Pic}{Pic}
\DeclareMathOperator{\Br}{Br}
\DeclareMathOperator{\inv}{inv}

\DeclareMathOperator{\Res}{Res}
\DeclareMathOperator{\Spec}{Spec}

\DeclareMathOperator{\tr}{tr}

\DeclareMathOperator{\Cores}{Cores}
\DeclareMathOperator{\Hom}{Hom}
\DeclareMathOperator{\Norm}{Norm}
\DeclareMathOperator{\et}{\acute{e}t}
\DeclareMathOperator{\proj}{proj}
\DeclareMathOperator{\Image}{Im}
\renewcommand{\epsilon}{\varepsilon}

\newcommand{\nequiv}{\not\equiv}
\newcommand{\sumstar}{\sideset{}{^\star} \sum}
\DeclareMathOperator{\prim}{prim}

\title{Cubic surfaces failing the integral Hasse principle}

\author{Julian Lyczak}
\address{Julian Lyczak, Department of Mathematical Sciences, University of Bath, Claverton Down, Bath, BA2 7AY, UK}
\email{jl4212@bath.ac.uk}

\author{Vladimir Mitankin}
\address{Vladimir Mitankin, Institute of Mathematics and Informatics, Bulgarian Academy of Sciences, Acad. G. Bonchev St., bl.8, 1113 Sofia, Bulgaria}
\address{Institut für Algebra, Zahlentheorie und Diskrete Mathematik,
	Leibniz Universität Hannover, Welfengarten 1, 30167 Hannover, Germany}
\email{v.mitankin@math.bas.bg}

\author{H. Uppal}
\address{H. Uppal, Department of Mathematical Sciences, University of Bath, Claverton Down, Bath, BA2 7AY, UK}
\email{hsu20@bath.ac.uk}

\date{\today}

\begin{document}
	\onehalfspacing
	
	\begin{abstract}
		We study the integral Brauer--Manin obstruction for affine diagonal cubic surfaces, which we employ to construct the first counterexamples to the integral Hasse principle in this setting. We then count in three natural ways how such counterexamples are distributed across the family of affine diagonal cubic surfaces and how often such surfaces satisfy integral strong approximation off $\infty$.
	\end{abstract}  
	
	\maketitle
	\setcounter{tocdepth}{1}
	\tableofcontents
	
	\section{Introduction}
	\label{sec:intro}
	This article is devoted to the study of integers represented by diagonal ternary cubic forms. A special case of this corresponds to the famous unsolved problem of which integers are the sum of three cubes, studied by Jacobi \cite{Dic66}, Mordell \cite{Mor42}, Heath-Brown \cite{HB92}, Colliot-Thélène and Wittenberg \cite{CTW12}, Booker \cite{B19}, Booker and Sutherland \cite{BS21} and Wang \cite{W23} among many others. Our main results show for the first time that local representations do not suffice for the existence of an integral representation, i.e. the integral Hasse principle does not hold in this setting. We construct the first examples of diagonal ternary cubic forms representing an integer over $\QQ$ but for which the integral Hasse principle fails. In fact, we give two infinite families of such examples in section \ref{sec:lower} and section \ref{sec: examples of IBMO}.  
	
	To set up the framework for our investigation, let $a_{0}, a_{1}, a_{2}, a_{3}$ be non-zero integers such that $\gcd(a_1, a_2, a_3) = 1$ and denote by $U$ the smooth affine surface over $\QQ$ given by
	\begin{equation}
		\label{eqn:U-main}
		U: \quad a_1u_1^3 + a_2u_2^3 + a_3u_3^3 = a_0.
	\end{equation}
	We fix an integral model $\sU$ of $U$ over $\ZZ$, defined by the same equation. Integral representations of $a_0$ by $a_1u_1^3 + a_2u_2^3 + a_3u_3^3$ correspond to integral points on $\sU$, whose set will be denoted by $\sU(\ZZ)$. A collection of local representations of $a_0$ at all places of $\QQ$ corresponds to a point in the adelic space $\sU(\AA_\ZZ) = U(\RR) \times \prod_{p \neq \infty} \sU(\Zp)$. A necessary condition for $\mathcal{U}(\mathbb{Z})\neq \emptyset$ is $\mathcal{U}(\mathbb{A}_{\mathbb{Z}})\neq\emptyset$, if this condition is also sufficient then the \emph{integral Hasse principle} holds. If for some prime $p$, we have $p^k \mid \gcd(a_1, a_2, a_3)$ for $k\in \mathbb{Z}_{>0}$ but $p^k \nmid a_0$, then clearly $\sU(\Zp) = \emptyset$ and questions about local-global principles are trivial. Moreover, $\gcd(a_1, a_2, a_3) =1$ implies $\gcd(a_0, a_1, a_2, a_3) = 1$ and thus ensures the existence of a $\ZZ$-model defined as above. On the other hand, $\sU$ is said to be an \emph{(integral) Hasse failure} if $\sU(\AA_\ZZ) \neq \emptyset$ but $\sU(\ZZ) = \emptyset$. If the closure of $\sU(\ZZ)$ in the \emph{finite integral adelic points} $\prod_{p \neq \infty} \sU(\Zp)$ coincides with this set, then \emph{integral strong approximation (ISA) off $\infty$} holds for $\sU$.
	
	The Brauer--Manin obstruction \cite{Man71} is a powerful tool to study local-to-global principles. Its wide applications and the conjecture of Colliot-Thélène \cite[p.~174]{CT03} claiming that it is capable of explaining all failures of the Hasse principle and weak approximation for a wide class of algebraic varieties testify for that. An integral version of this tool was developed by Colliot-Thélène and Xu in \cite{CTX09}. They defined an integral Brauer--Manin set $\sU(\AA_\ZZ)^{\Br}$ by taking those integral adeles that pair to zero with the Brauer group $\Br U$ of $U$ and thus obtained a chain of inclusions $\sU(\ZZ) \subseteq \sU(\AA_\ZZ)^{\Br} \subseteq \sU(\AA_\ZZ)$. A Brauer--Manin obstruction to the integral Hasse principle is present if $\sU(\AA_\ZZ) \neq \emptyset$ but $\sU(\AA_\ZZ)^{\Br} = \emptyset$ forcing $\sU(\ZZ) = \emptyset$. It obstructs strong approximation off $\infty$ if the image of $\mathcal{U}(\mathbb{A}_{\mathbb{Z}})^{\Br}$ under the projection to $\prod_{p\neq\infty}\mathcal{U}(\mathbb{Z}_{p})$ is a strict subset.

	\subsection*{Main results}
	To study the Brauer--Manin obstruction on $U$, as given as in \eqref{eqn:U-main}, we first need to examine $\Br U$, which we calculate in full generality for the first time. We relate $\Br U$ to the Brauer group of its compactification $X$. The group $\Br X$ modulo constants was determined in {\cite[\S1, Prop.~1]{CTKS87}} over fields containing a primitive third root of unity. It follows from their computation that over $\mathbb{Q}$ the group $\Br X$ modulo constants is either trivial or $\mathbb{Z}/3\mathbb{Z}$.
	
	\begin{theorem}
		\label{thm: Brauer group thm}
		Assume that $a_0, a_1, a_2, a_3 \in \mathbb{Q}^{\ast}$. Then the algebraic part of $\Br U$ is isomorphic to $\Br X$. Moreover,
		\[
		\Br U =
		\begin{cases} 
			\Br X \oplus \mathbb{Z}/2\mathbb{Z} \ &\text{if}\ a_{1}a_{2}a_{3} \equiv 2 \bmod \QQ^{\ast 3},\\
			\Br X &\text{otherwise}.
		\end{cases}
		\]
	\end{theorem}

	Building on the foundational work of Colliot-Thélène and Wittenberg \cite{CTW12}, our approach in the proof of Theorem~\ref{thm: Brauer group thm} yields a comprehensive understanding of the Brauer group of all smooth diagonal affine cubic surfaces. We perform a fine analysis of the properties of Mordell curves, which is additionally needed in the general setting we work in. The method presented here is, in fact, capable of tackling more general surfaces. Specifically, it can be adapted to study Brauer groups of affine surfaces that are the complement of a smooth irreducible anticanonical curve on a geometrically rational surface.
	
	We use the information obtained in Theorem~\ref{thm: Brauer group thm} to facilitate a detailed examination of the arithmetic of $U$ by employing the integral Brauer--Manin obstruction. As a result, for the first time failures of the integral Hasse principle on these surfaces are exhibited. We construct two rich infinite families of such examples in section \ref{sec:lower} and section \ref{sec: examples of IBMO}. This was not previously possible due to the Brauer groups of these surfaces not being known in such a generality. 
	
	To explore the frequency of integral Hasse failures and to measure how often integral strong approximation off $\infty$ holds in the family \eqref{eqn:U-main}, we either vary $a_0$, the coefficients of the cubic form $a_1$, $a_2$, $a_3$, or all $a_0$, $a_1$, $a_2$, $a_3$ in a equal-sided box. In view of Remark~\ref{rem:els}, the number of everywhere locally soluble $\sU$ up to height $B$ in the three counting problems is of magnitude $B$, $B^3$, $B^4$, respectively. We give in Propositions~\ref{prop:p divides a_0} and \ref{prop:p divides a_1 but} a sufficient criterion for the lack of a Brauer--Manin obstruction to the integral Hasse principle and for the presence of an Brauer--Manin obstruction to integral strong approximation off $\infty$. Counting surfaces failing this criterion allows us to obtain upper bounds for the amount of Hasse failures and simultaneously to estimate the number of surfaces satisfying strong approximation off $\infty$ in all three natural counting problems.
	
	We begin with the analogue of the sum of three cubes question, i.e. we fix $a_1, a_2, a_3$ and we vary $a_0$. Set $\ZZ_{\prim}^n$ for the set of $n$-tuples in $\ZZ^n$ with non-zero coprime coordinates. If $\ZZ_{\neq 0}$ stands for the non-zero integers, for any real $B\ge 1$ and $(a_1, a_2, a_3) \in \ZZ_{\prim}^3$ define
	\[
	\begin{split}
		N_{a_1, a_2, a_3}(B) 
		&= \#\left\{ a_0 \in [-B, B] \cap \ZZ_{\neq 0} \ : \ \sU(\AA_\ZZ) \neq \emptyset \text{ but } \sU(\AA_\ZZ)^{\Br} = \emptyset \right\}, \\
		N_{a_1, a_2, a_3}^{'}(B)
		&= \#\left\{ a_0 \in [-B, B] \cap \ZZ_{\neq 0} \ : \ \sU(\AA_\ZZ) \neq \emptyset \text{ and } \sU \text{ satisfies ISA off $\infty$} \right\}.
	\end{split}
	\]
	Our next result provides upper bounds for these quantities.
	
	\begin{theorem}
		\label{thm:N-a_123(B)}
		Assume that $a_1a_2a_3 \nequiv 2 \bmod \QQ^{\ast 3}$. We then have
		\[
		N_{a_1, a_2, a_3}(B), \ N_{a_1, a_2, a_3}^{'}(B) \ll_{a_1 a_2 a_3} B^{1/3},
		\] 
		as $B$ goes to infinity.
	\end{theorem}
	
	Theorem~\ref{thm:N-a_123(B)} shows how rare Hasse failures are if the cubic form is fixed, thus revealing the difficulty of finding explicit examples of them. In fact, $N_{a_1, a_2, a_3}(B)$ is zero for the two specific families $(a_{1}, a_{2}, a_{3}) = (1,1,1)$ and $(1, 1, 2)$ studied by Colliot-Thélène and Wittenberg \cite{CTW12}. Hence, finding the magnitude of $N_{a_1, a_2, a_3}(B)$, or even a lower bound for it, amounts to selecting specific choices of $(a_{1}, a_{2}, a_{3})$ for which Hasse failures exist. Such choices are extremely rare, as we shall see in the next results.
	
	We continue by varying the cubic form while $a_0 \neq 0$ stays fixed. Let
	\[
	\begin{split}
		N_{a_0}(B) 
		&= \# \left\{ (a_1, a_2, a_3) \in [-B, B]^3 \cap \ZZ_{\prim}^3 \ : \ \sU(\AA_\ZZ) \neq \emptyset \text{ but } \sU(\AA_\ZZ)^{\Br} = \emptyset \right\}, \\
		N_{a_0}^{'}(B) 
		&= \# \left\{ (a_1, a_2, a_3) \in [-B, B]^3 \cap \ZZ_{\prim}^3 \ : \ \sU(\AA_\ZZ) \neq \emptyset \text{ and } \sU \text{ satisfies ISA off $\infty$} \right\}.
	\end{split}
	\]
	Our methods establish upper bounds for these quantities, given in the next theorem.
	
	\begin{theorem}
		\label{thm:N-a_0(B)}
		We have
		\[
		N_{a_0}(B), \ N_{a_0}^{'}(B) \ll_{a_0} B^{3/2},
		\] 
		as $B$ goes to infinity.
	\end{theorem}
	
	Lastly, we vary all four coefficients of $\sU$. For this purpose let
	\[
	\begin{split}
		N(B) 
		&= \# \left\{ (a_0, a_1, a_2, a_3) \in [-B, B]^4 \cap \ZZ_{\prim}^4 \ : \ \sU(\AA_\ZZ) \neq \emptyset \text{ but } \sU(\AA_\ZZ)^{\Br} = \emptyset \right\}, \\
		N^{'}(B) 
		&= \# \left\{ (a_0, a_1, a_2, a_3) \in [-B, B]^4 \cap \ZZ_{\prim}^4 \ : \ \sU(\AA_\ZZ) \neq \emptyset \text{ and } \sU \text{ satisfies ISA off $\infty$} \right\}.
	\end{split}
	\]
	Our final result delivers upper and lower bounds for $N(B)$, of the same magnitude modulo a small power of $\log B$, and an upper bound for $N^{'}(B)$.
	\begin{theorem}
		\label{thm:N(B)}
		We have
		\[
		\begin{split}
			\frac{B^2}{\log B} \ll N(B) &\ll B^2 (\log B)^6 \quad \text{and} \quad
			N^{'}(B) \ll B^2 (\log B)^2,
		\end{split}
		\] 
		as $B$ goes to infinity.
	\end{theorem}
	
	We construct in section \ref{sec:lower} and in section \ref{sec: examples of IBMO} the first examples of diagonal affine cubic surfaces with a Brauer--Manin obstruction to the integral Hasse principle. Moreover, all of the surfaces in section \ref{sec:lower} and section \ref{sec: examples of IBMO} have a non-empty set of rational points (see Remarks~\ref{rem:Q-point} and \ref{remark: CT conj implies that Ulpq has a rational point}) and thus our examples of integral Hasse failures do not follow trivially from Hasse failures for rational points. Colliot-Th\'{e}l\`ene, Kanevsky and Sansuc constructed an infinite family of cubic surfaces, failing the Hasse principle for rational points in \cite[\S 7, Prop.~5]{CTKS87}. Each of their surfaces produces integral Hasse failures by taking away exactly one of the hyperplanes corresponding to the zero locus of a coordinate. It is relatively easy to see that the number of failures coming from Colliot-Th\'{e}l\`ene, Kanevsky and Sansuc's family counted by $N(B)$ is at most $B/(\log B)^2$. Hence, our lower bound does not follow from the results in \cite{CTKS87}.
	
	The lower bound of Theorem~\ref{thm:N(B)} is obtained by counting surfaces of the family featured in section \ref{sec:lower}. At the same time, the counter-examples to the integral Hasse principle appearing in section \ref{sec: examples of IBMO} are interesting in their own right. Our approach in section \ref{sec: examples of IBMO} builds on the work of Colliot-Th\'{e}l\`ene, Kanevsky and Sansuc \cite{CTKS87} and does not require any knowledge of explicit representatives of Brauer elements, unlike in section \ref{sec:lower}. This is particularly handy, as explicit representatives of Brauer elements are genuinely very hard to get a hold of.
	
	Finally, we note that the number of surfaces considered in Theorem~\ref{thm:N-a_0(B)}, whose transcendental Brauer group is non-trivial, is negligible compared to the upper bounds established there. This is shown in Proposition~\ref{prop:trans}. At the same time, all possible failures of the integral Hasse principle and integral strong approximation off $\infty$ counted in $N_{a_1, a_2, a_3}(B)$, $N_{a_1, a_2, a_3}^{'}(B)$ and in $N(B)$, $N^{'}(B)$ may come from transcendental Brauer elements.
	
	\subsection*{Integral points in families}
	In recent years the quantitative study of arithmetic properties in families has rapidly expanded its scope of investigation. Several papers focus on the relevance of the Brauer--Manin obstruction for rational points in families; \cite{BB14}, \cite{Rom19} for Châtelet surfaces, \cite{GLN22}, \cite{San23} for certain classes of K3 surfaces, \cite{MS22} for quartic del Pezzo surfaces, \cite{BB14b} for coflasque tori, etc. It was shown in the proof of \cite[Thm.~1.6]{BBL16} that under certain geometrical assumptions weak approximation fails and at the same time the Brauer--Manin obstruction does not obstruct the Hasse principle for rational points 100\% of the time, even if single members of the family might have a Brauer--Manin obstruction to the Hasse principle. No such framework exists for the study of integral points in families yet. There are very few examples studied in the literature; \cite{Mit17}, \cite{San22}, \cite{MNS22} studied affine quadrics surfaces, \cite{Ber17}, \cite{Mit20}, \cite{BST22} considered affine generalized Châtelet surfaces and Markoff surfaces were treated in \cite{GS22}, \cite{LM20}, \cite{CTWX20}.
	
	The surface $U$ is the complement of an anticanonical curve on a del Pezzo surface of degree 3. In particular, it is a log K3 surface in the sense of \cite[Def. 2.4]{H17}. The integral Brauer--Manin obstruction was successfully employed to study integral points for such types of log K3s in \cite{CTW12}, \cite{U23} for affine cubic surfaces, \cite{JS17}, \cite{L23} for affine del Pezzo surfaces of degree $4$ and $5$ and \cite{LM20},\cite{CTWX20} for Markoff type log K3. The investigation on the frequency of Hasse failures for Markoff type log K3s was pioneered in \cite{GS22} resulting in a wide open conjecture for that quantity. This conjecture was then explored in \cite{LM20}, \cite{CTWX20} by measuring the amount of Hasse failures with or without a Brauer--Manin obstruction. No integral Brauer--Manin obstruction had been exhibited in the diagonal cubic case before the present work.
	
	\subsection*{Outline of the paper} This article is organised as follows. Section~\ref{sec: Integral Brauer--Manin obstruction} recalls the integral Brauer--Manin obstruction. Section~\ref{sec: Brauer group} is dedicated to finding the Brauer group of the affine surfaces (\ref{eqn:U-main}) and the proof of Theorem~\ref{thm: Brauer group thm}. We describe explicit generators of $\Br U / \Br \QQ$ in special cases in section~\ref{sec:2-coeff}. In section~\ref{sec: computing inv} we describe how to compute the integral Brauer--Manin set for a generic choice of diagonal affine cubic surface. In section~\ref{sec:3-coeff} we study the integral Brauer--Manin obstruction for a three coefficient family which is later used in section~\ref{sec:lower} to prove the lower bound in Theorem~\ref{thm:N(B)}. Section~\ref{sec:upper} is reserved for the upper bounds in Theorems~\ref{thm:N-a_123(B)}, \ref{thm:N-a_0(B)} and \ref{thm:N(B)}. Finally, the results of section~\ref{sec: computing inv} are implemented in section~\ref{sec: examples of IBMO}, where we give a family of surfaces failing the integral Hasse principle without the need of knowing explicit representatives of Brauer elements.
	
	\subsection*{Notation}
	We reserve $K$ for a number field with ring of integers $\sO_K$. Throughout $\Omega_K$ will denote the set of places of $K$. For $v \in \Omega_K$ we shall use $K_v$ for the completion of $K$ at $v$ and $\sO_v$ for its ring of integers with the convention $\sO_v = K_v$ for all infinite places. If $\frakp$ is a prime ideal above $p$ in $K$, we shall denote by $K_{\frakp}$ the corresponding extension of $\Qp$, and by $\sO_{\frakp}$ the ring of integers of $K_{\frakp}$ to emphasise the fixed choice of $\frakp$. We fix $\omega \in \bar{\QQ}$ such that $\omega^2 + \omega + 1 = 0$ and $k:=\mathbb{Q}(\omega)$. 
	
	Let $V$ be a variety over $K$. For an $\sO_K$-integral model $\sV$ of $V$, we define the \emph{integral adelic set} of $\sV$ as	$\mathcal{V}(\mathbb{A}_{\sO_K}):= \prod_{v \in \Omega_K} \mathcal{V}(\sO_v)$. We denote by $V(\mathbb{A}_{K})$ the set of adelic points on $V$, that is the restricted product $V(\AA_K) = \prod_{v \in \Omega_K}' (V(K_v), \sV(\sO_v))$. If $S$ is a scheme over a ring $R$ and $R\rightarrow T$ is a morphism, we denote by $S_{T}$ the base change $S_{T}:=S\times_{\Spec{R}}\Spec T$. For a field $F$ with a fixed algebraic closure $\bar F$ we will write $\bar{S}$ for the base change $\bar{S}:=S\times_{\Spec F}\Spec \bar{F}$ for an $F$-scheme $S$. 
	
	We fix a choice of \emph{(smooth) compactification $X \subset \PP_{\QQ}^3$ of $U$}, given in \eqref{eqn:U-main}, which is
	\[
	X: \quad a_1x_1^3 + a_2x_2^3 + a_3x_3^3 = a_0x_0^3.
	\]
	Furthermore, we fix $D$ for the divisor on $X$ defined by $D:=\{x_{0}=0\}$, unless otherwise stated, and hence $U = X \setminus D$. Integral models of $U$, $X$ and $D$ will be denoted by $\sU$, $\sX$ and $\sD$, respectively. They are assumed to be given by the defining equation of $U$, $X$ and $D$, respectively, unless otherwise stated. 

	\begin{acknowledgements}
		We would like to thank Olivier Wittenberg for many useful conversations and helping with the proof of Proposition~\ref{prop:2-torsion}. We thank Lyubcho Baltadjiev, Tim Browning, Daniel Loughran, Ross Paterson and Tim Santens for enlightening discussions. We are also grateful to the anonymous referees for their helpful comments. The first named author was supported by UKRI MR/V021362/1. The second named author was partially supported by Horizon Europe 2023 MSCA postdoctoral fellowship 101151205 -- GIANT, funded by the European Union, by grant DE 1646/4-2 of the Deutsche Forschungsgemeinschaft and by scientific program ``Enhancing the Research Capacity in Mathematical Sciences (PIKOM)'', No. DO1-67/05.05.2022 of the Ministry of Education and Science of Bulgaria.
	\end{acknowledgements}
	
	\section{Integral Brauer--Manin obstruction}
	\label{sec: Integral Brauer--Manin obstruction}
	In this section we recall the integral Brauer--Manin obstruction as introduced by Colliot-Thélène and Xu \cite[\S1]{CTX09}.
	\begin{definition}[Brauer--Grothendieck group]
		Let $V$ be a scheme over a field $F$. The cohomological Brauer group of $V$ is the second étale cohomology group $\Br V := \HH^{2}_{\et}(V,\mathbb{G}_{m})$. We have the filtration 
		\[
		0 \subseteq \Br_{0} V \subseteq \Br_{1} V \subseteq \Br V,
		\]
		where $\Br_{0} V:=\text{Im}\(\Br F \rightarrow \Br V\)$ and $\Br_{1} V:=\ker(\Br V \rightarrow \Br \bar{V})$. We call $\Br_{1} V$ the \emph{algebraic Brauer group of $V$} and $\Br V/\Br_{1} V$ the \emph{transcendental Brauer group of $V$}.
	\end{definition} 
	
	For all places $v \in \Omega_{K}$ there exists a canonical injective map
	$\inv_{v}:\Br K_{v} \rightarrow \mathbb{Q}/\mathbb{Z}$ \cite[Def.~13.1.7]{CS21}, whose image is $\{0,1/2\}$ if $v$ is a real place, $0$ if it is a complex place and is an isomorphism if $v$ is finite \cite[Thm.~13.1.8]{CS21}. 
	
	Let $\mathcal{V}$ be a separated scheme of finite type over $\sO_K$ and let $V$ be its base change to $K$. For every $\alpha \in \Br V$ there exists a finite set of places $S_{\alpha} \subseteq \Omega_K$ such that the invariant map composed with evaluation $\inv_v \alpha : \sV (\sO_v) \rightarrow \QQ/\ZZ$ is zero for all $v \not\in S_{\alpha}$, \cite[\S1]{CTX09}. This shows that the Brauer--Manin pairing, as given below, is well defined:
	\[
	\begin{split}
		\mathcal{V}(\mathbb{A}_{\sO_K}) \times \Br V &\longrightarrow \mathbb{Q}/\mathbb{Z}, \\ 
		\left((x_{v})_{v},\alpha\right) &\longmapsto \sum\limits_{v\in \Omega_K} \inv_{v}(\alpha(x_{v})).
	\end{split}
	\] 
	Define the \emph{integral Brauer--Manin set} $\mathcal{V}(\mathbb{A}_{\sO_{K}})^{\Br}$ as the left kernel of this paring. Since the following diagram commutes:
	\begin{center}
		\begin{tikzcd}
			& & \sV \(\sO_K\) \arrow[r,hook] \arrow[d,"\alpha"] & \sV\left(\AA_{\sO_K}\right) \arrow[d,"\alpha"] & & \\
			& 0 \arrow[r] & \Br K \arrow[r] & \bigoplus\limits_{v \in \Omega_K} \Br K_v \arrow[r,"\sum_v \inv_v"] & \QQ/\ZZ \arrow[r] & 0,
		\end{tikzcd}
	\end{center}
	where exactness of the bottom row is implied by the Albert--Brauer--Hasse--Noether theorem and class field theory, we get a chain of inclusions
	\[
	\sV(\sO_K) 
	\subseteq \sV (\AA_{\sO_K})^{\Br}
	\subseteq \sV(\AA_{\sO_K}).
	\]
	
	\begin{definition}
		\label{def:BMO}
		There is a \emph{Brauer--Manin obstruction to the integral Hasse principle for $\sV$}, or $\sV$ is a \emph{(integral) Hasse failure}, if $\sV(\AA_{\sO_K}) \neq \emptyset$ but $\sV (\AA_{\sO_K})^{\Br} = \emptyset$. If $B \subset \Br V$, we say that there is a \emph{Brauer--Manin obstruction to the integral Hasse principle for $\sV$ coming from $B$}, if $\sV(\AA_{\sO_K}) \neq \emptyset$ but $\sV (\AA_{\sO_K})^{B} = \emptyset$.
	\end{definition}
	
	Even if there is no Brauer--Manin obstruction to the integral Hasse principle, the following two notions can be obstructed by the Brauer--Manin construction. This is seen in Sections \ref{sec:upper} and \ref{sec: examples of IBMO}.
	
	\begin{definition}
		\label{def:SA-and-WA}
		We say $V$ satisfies \emph{strong approximation} (respectively \emph{weak approximation}) if the diagonal image $V(K)\hookrightarrow V(\AA_{K})$ (respectively $V(K)\hookrightarrow \prod_{v\in \Omega_{K}} V(K_{v})$) is dense. We say $\sV$ satisfies \emph{integral strong approximation off $\infty$} if the diagonal image $\sV(\sO_K) \hookrightarrow \proj_{\infty}\sV(\AA_{\sO_{K}})$ is dense, where $\proj_{\infty}:\mathcal{V}(\mathbb{A}_{\mathcal{O}_{K}})\rightarrow \prod_{v\nmid \infty}\mathcal{V}(\mathcal{O}_{v})$ means the projection to the finite adeles $\prod_{v\nmid \infty}\mathcal{V}(\mathcal{O}_{v})$. 
	\end{definition}
	
	\begin{definition}
		\label{def:BMO-SA-WA}
		There is a \emph{Brauer--Manin obstruction to strong approximation} (respectively \emph{integral strong approximation off $\infty$}) if $V(\AA_{K})^{\Br}$ (respectively $\proj_{\infty} \sV(\AA_{\sO_K})^{\Br}$) is a strict subset of the adeles $V(\mathbb{A}_{K})$ (respectively finite integral adeles $\prod_{v\nmid \infty}\mathcal{V}(\mathcal{O}_{v})$).
	\end{definition}
	
	\begin{remark}
		\label{rem:els}
    Questions about local-global principles are non-trivial only for varieties with a non-empty adelic space. For completeness, we briefly discuss the amount of everywhere locally soluble surfaces in the three counting problems set-up in the introduction here. This analysis is not needed in the remaining of the paper. With notation as in the introduction, for any real $B \ge 1$ let
  	\[
  	\begin{split}
  		N_{a_1, a_2, a_3}^{\text{ELS}}(B) 
  		&= \#\left\{ a_0 \in [-B, B] \cap \ZZ_{\neq 0} \ : \ \sU(\AA_\ZZ) \neq \emptyset \right\} \text{ for } (a_1, a_2, a_3) \in \ZZ_{\prim}^3, \\
  		N_{a_0}^{\text{ELS}}(B) 
  		&= \# \left\{ (a_1, a_2, a_3) \in [-B, B]^3 \cap \ZZ_{\prim}^3 \ : \ \sU(\AA_\ZZ) \neq \emptyset \right\} \text{ for } a_0 \in \ZZ_{\neq 0}, \\
  		N^{\text{ELS}}(B) 
  		&= \# \left\{ (a_0, a_1, a_2, a_3) \in [-B, B]^4 \cap \ZZ_{\prim}^4 \ : \ \sU(\AA_\ZZ) \neq \emptyset \right\}.
  	\end{split}
  	\]

		Then, the limits
		\[
			\lim_{B \to \infty} \frac{N_{a_1, a_2, a_3}^{\text{ELS}}(B)}{B}, \quad
			\lim_{B \to \infty} \frac{N_{a_0}^{\text{ELS}}(B)}{B^3}, \quad
			\lim_{B \to \infty} \frac{N^{\text{ELS}}(B)}{B^4}
		\]
		exist and each of them equals a positive constant. This may be verified, for example, with the help of \cite[Lem.~3.1]{BBL16}, which approximates the number of lattice points in an archimedean region subject to local conditions at each finite place. To apply the result one needs to check three assumptions. The second assumption asks the region and the local conditions to have a positive measure, and this is easily seen along the lines of the proof of \cite[Thm.~2.2]{BBL16}. The first conditions requires that the boundary of the region at each place has measure $0$, which follows from the proof of \cite[Thm.~2.2]{BBL16} modified for $\Zp$-points on $\sU$. The third condition imposes that the combined condition at all finite places are not too restrictive. For our application this follows from a completely analogous analysis to the one featured in the proof of \cite[Thm.~1.1]{Mit17} as a prime has to divide at least two of the $a_i$ for the lack of $\Zp$-points.
	\end{remark}

	\section{Brauer groups of diagonal affine cubic surfaces}
	\label{sec: Brauer group} 
	This section is dedicated to finding the Brauer group of the diagonal affine cubic surfaces $U$ over $\QQ$ given in \eqref{eqn:U-main}. Recall that $X$ is its compactification and $D$ is the boundary divisor of $X$ such that $U = X \setminus D$.
	
	\subsection{Algebraic Brauer group}\label{subsec: algebaric Brauer group}
	In this subsection we compute the algebraic Brauer group of $U$ over a number field $K$ that does not contain a primitive third root of unity.
	
	\begin{proposition}\label{prop: algebaric Brauer group for affine is the same as projective}
		Let $U\to X$ over $K$ be a smooth compactification of the affine diagonal cubic given by \eqref{eqn:U-main}.
		\begin{enumerate}[label=\emph{(\arabic*)}]
			\item The natural map $\Br X\rightarrow \Br_{1}U$ is an isomorphism,
			\item
			$
			\Br_{1}U/\Br_{0}U =
			\begin{cases}
				0 &\text{ if a cross ratio } a_{i}a_{j}/a_{h}a_{\ell}\in K^{\ast 3}, \\
				\mathbb{Z}/3\mathbb{Z} &\text{otherwise}.
			\end{cases}
			$
			\item $\Br_{1} U/\Br_{0}U$ is generated by $\mathcal{A}':=\Cores_{K(\omega)/K}\mathcal{A}$ where $\mathcal{A}$ is an element of $\Br_{1}U_{K(\omega)}/\Br_{0}U_{K(\omega)}$.
		\end{enumerate}
	\end{proposition}
	
	In general, one can compute $\Br_1 U$ modulo constants following \cite[Prop.~2.1]{BL19}, since the divisor $D$ is irreducible. Using the fact $\HH^{0}(\bar{X},\mathbb{G}_{m})=\bar{K}^{\ast}=\HH^{0}(\bar{U},\mathbb{G}_{m})$ and $\HH^{3}(K,\bar{K}^{\ast})=0$ the Hochschild--Serre spectral sequence gives us natural group isomorphisms
	\[
	\Br X/\Br_{0}X \cong \HH^{1}(K,\Pic \bar{X})\quad \text{ and } \quad \Br_{1}U/\Br_{0}U \cong \HH^{1}(K,\Pic \bar{U}).
	\]
	As $D$ is irreducible we have the exact sequence
	\[
	0\rightarrow \mathbb{Z}\xrightarrow[]{i} \Pic \bar{X}\xrightarrow[]{\pi}\Pic \bar{U}\rightarrow 0,
	\]
	where $i$ is defined by $1 \mapsto [D]=-K_X$. The canonical class $K_X$ is primitive in $\Pic \bar{X} \cong \mathbb Z^7$, so $\Pic \bar{U}$ and $\Pic \bar{X}$ are both torsion-free. We will use the following notion to descend to a finite field extension.
	\begin{definition}
		Let $S$ be a smooth cubic surface over a number field $K$. The \emph{splitting field} of $S$ is the minimal normal extension $L$ of $K$ such that $\Pic S_{L}\xrightarrow{\sim}\Pic \bar{S}$.
	\end{definition}

	As $\Gal(\bar{L}/L)$ is a profinite group acting trivially on the torsion-free module $\Pic U_L$ the inflation-restriction sequence gives us $\Br_{1}U/\Br_{0}U \cong \HH^{1}(\Gal(L/K),\Pic U_L)$ where $L$ is the splitting field of $X$. In general, $\Gal(L/K)$ can only act in finitely many different ways on $\Pic X_L$. This leads to a computational verification of Proposition \ref{prop: algebaric Brauer group for affine is the same as projective}.

	\begin{proof}[Proof of Proposition \ref{prop: algebaric Brauer group for affine is the same as projective}]
		Let $X_{\alpha_{1},\alpha_{2},\alpha_{3}}$ be the surface \[
		X_{\alpha_{1},\alpha_{2},\alpha_{3}}:x_{0}^3+\alpha_{1}x_{1}^3+\alpha_{2}x_{2}^3+\alpha_{3}x_{3}^3=0\subset \mathbb{P}^{3}
		\] over the field $E:=K(\alpha_{1},\alpha_{2},\alpha_{3})$ where $\alpha_{i}$ are purely transcendental elements, the splitting field of this surface is given by the extension $E':=E(\omega,\sqrt[3]{\alpha_{1}},\sqrt[3]{\alpha_{2}},\sqrt[3]{\alpha_{3}})$ where $\omega\in \bar{K}$ such that $\omega^2+\omega+1=0$. There exists a primitive element $\beta\in E'$ such that $E'=E(\beta)$, with minimal polynomial $f\in E[x]$. For any $L/K$ obtained by specialising $\alpha_{i}$ in $K$, where $f$ specialises to a separable polynomial, gives an embedding $\Gal(L/K)\hookrightarrow \Gal(E'/E)$ \cite[61.1]{VDW91}. 
		This shows that the splitting field for $X$ is given by $L=K\left(\omega,\sqrt[3]{a_{1}/a_{0}},\sqrt[3]{a_{2}/a_{0}},\sqrt[3]{a_{3}/a_{0}}\right)$ and $\Gal(L/K)\hookrightarrow \Gal(E(\omega,\sqrt[3]{\alpha_{1}},\sqrt[3]{\alpha_{2}},\sqrt[3]{\alpha_{3}})/E)$. We get a chain of inclusions
		\[
		\Gal(K(\omega)/K)\hookrightarrow \Gal(L/K)\hookrightarrow \Gal(E'/E)\cong C_{3}^3 \rtimes C_{2}
		\] which act compatibly on $\Pic \bar{X}\xrightarrow[]{\sim} \Pic \bar{X}_{\alpha_{1},\alpha_{2},\alpha_{3}}$. From here we can determine the image of the representation $\rho \colon \Gal(\bar{K}/K)\rightarrow W(E_{6})$ induced by the action of $\Gal(\bar{K}/K)$ on $\Pic \bar{X}$. Note that this action is trivial for element in $\Gamma_L := \Gal(\bar{K}/L)$. The action is determined by the image $N$ of $\rho$, which is isomorphic to $\Gal(\bar{K}/K)/\Gamma_L$. We enumerate all such subgroups $N$ of the Weyl group $W(E_{6})$, which are isomorphic to $C_{3}^3 \rtimes C_{2}$, up to conjugacy. Next, we compute $\HH^1(N', (\Pic \bar{U})^{\Gamma_L})$ for all subgroups $N' \subseteq N$. There are 8 even order subgroups and 4 odd order subgroups where $\HH^{1}(N',(\Pic \bar{X})^{\Gamma_L})=\HH^{1}(N',(\Pic \bar{U})^{\Gamma_L})$ and 4 odd order subgroups where $\HH^{1}(N',(\Pic \bar{X})^{\Gamma_L}) \neq\HH^{1}(N',(\Pic \bar{U})^{\Gamma_L})$.
		
		As $K$ does not contain a primitive third root of unity $2\mid \#N'$, hence \[
		\Br X/\Br_{0}X\cong \HH^{1}(N',(\Pic \bar{X})^{N'}) \xrightarrow[]{\sim}\HH^{1}(N',(\Pic \bar{U})^{N'})\cong \Br_{1}U/\Br_{0}U.
		\] As $\Br_{0}X\xrightarrow[]{\sim} \Br_{0}U$ we have the following commutative diagram \[
		\begin{tikzcd}
			0 \arrow[r] & \Br_{0}X \arrow[rr] \arrow[dd, "\cong"'] &  & \Br X \arrow[dd] \arrow[rr] &  & \Br X/\Br_{0}X \arrow[dd, "\cong"] \arrow[r] & 0 \\
			&                               &  &                              &  &                             &   \\
			0 \arrow[r] & \Br_{0}U \arrow[rr]                  &  & \Br_{1}U \arrow[rr]                 &  & \Br_{1}U/\Br_{0}U \arrow[r]                 & 0
		\end{tikzcd}
		\]
		By the snake lemma we deduce that $\Br X\xrightarrow[]{\sim} \Br_{1}U$.
		
		For the last statement we use that $\Br X_{K(\omega)}/\Br_0 X_{K(\omega)}$ is either $0$, $\ZZ/3\ZZ$ or $(\ZZ/3\ZZ)^2$ \cite[\S1, Prop. 1]{CTKS87}. Since both $\Br X/\Br_0 X$ and $\Br X_{K(\omega)}/\Br_0 X_{K(\omega)}$ are $3$-torsion, and $[K(\omega)\colon K]=2$ we apply \cite[Thm.~3.8.5]{CS21}, which gives that $\Res_{K(\omega)/K}$ is an isomorphism
		\[
		\Res_{K(\omega)/K}:\Br X/\Br_0 X \to \left(\Br X_{K(\omega)}/\Br_0 X_{K(\omega)}\right)^{\text{Gal}(K(\omega)/K)},
		\]
		whose inverse is given by $-\Cores_{K(\omega)/K}$. Hence a generator $\mathcal A' \in \Br X/\Br_0 X$ corresponds to a unique Galois invariant element $\mathcal A \in \Br X_{K(\omega)}/\Br_0 X_{K(\omega)}$.
	\end{proof}
	
	\begin{remark}
		The condition that $K$ does not contain a primitive third root of unity is necessary in Proposition \ref{prop: algebaric Brauer group for affine is the same as projective}. Consider the cubic surface \[
		X'\colon \quad x_{0}^3+x_{1}^3+x_{2}^3+ax_{3}^3=0\subset \mathbb{P}^{3}_{\mathbb{Q}(\omega)}.
		\] where $a$ is cube-free. In this case $\Br X'/\Br \mathbb{Q}(\omega)\cong (\mathbb{Z}/3\mathbb{Z})^2$ \cite[\S1, Prop. 1]{CTKS87}. By \cite[Lem. 2.2]{BL19} for any smooth and irreducible anticanonical curve $D'$ the surface $U':=X'\backslash D'$ has $\Br_{1}U'/\Br \mathbb{Q}(\omega)\cong (\mathbb{Z}/3\mathbb{Z})^3$.
		
	\end{remark}
	\subsection{Transcendental Brauer group} In Proposition \ref{prop: algebaric Brauer group for affine is the same as projective} we showed that $\Br _{1}U=\Br X$. To establish Theorem \ref{thm: Brauer group thm}
	\[
	\Br U =
	\begin{cases} 
		\Br X \oplus \mathbb{Z}/2\mathbb{Z} \ &\text{if}\ a_{1}a_{2}a_{3} \equiv 2\ \text{mod}\ \mathbb{Q}^{\ast 3},\\
		\Br X &\text{otherwise},
	\end{cases}
	\]
	we need to determine the transcendental Brauer group of $U$. For this we adapt some results by Jahnel and Schindler \cite[Thm.~4.9, Cor 4.10, Rem 4.11]{JS17} on degree four del Pezzo surfaces to geometrically rational surfaces. Their proofs can be used almost verbatim in this situation. From there we use a method of Colliot-Thélène and Wittenberg \cite[\S 3,5]{CTW12}.

	\begin{lemma}\label{lemma: Brauer group over algebaric closure}
		Let $F$ be a field of characteristic $0$ and $S$ a smooth, geometrically rational projective surface over $F$. Denote by $H$ a hyperplane over $F$ and $V$ the affine surface $V:=S\backslash H$. Suppose $C:=H \cap S$ is smooth, then the natural morphism \[
		\Br(\bar{V})\rightarrow \HH^{1}_{\et}(\bar{C},\mathbb{Q}/\mathbb{Z}).
		\] is an isomorphism.
	\end{lemma}
	\begin{proof}
		As $\bar{S}$ is rational $\HH^{2}_{\et}(\bar{S},\mathbb{G}_{m})=\HH^{3}_{\et}(\bar{S},\mathbb{G}_{m})=0$. Then the statement follows from Grothendieck's purity theorem \cite[Thm.~3.7.1]{CS21}.
	\end{proof}
	
	\begin{proposition}\label{prop: isogenies and transcendental elements}
		Let $V$ and $C$ be as in Lemma \ref{lemma: Brauer group over algebaric closure}. Then there is an injection
		\[
		(\Br V/ \Br_{1} V)[n] \hookrightarrow \Hom(J(C)(\bar{F})[n],\mathbb{Q}/\mathbb{Z})^{{\Gal(\bar{F}/F)}}.
		\]
		In particular, if $C$ is a genus one curve then a class of order $n$ in $\Br V/\Br_{1} V$ induces an $F$-rational $n$-isogeny
		\[
		J(C) \twoheadrightarrow C'
		\]
		to an elliptic curve $C'$ with an $F$-point of order $n$.
	\end{proposition}
	
	\begin{proof}
		From the Kummer sequence of étale sheaves we can deduce $\HH_{\et}^{1}(\bar{C},\mu_{n})\cong \Pic(\bar{C})[n]$ and $\HH^{2}_{\et}(\bar{C},\mu_{n})\cong \mathbb{Z}/n\mathbb{Z}$. As $C$ is a smooth curve \[
		\Pic(\bar{C})[n]\cong \Pic^0(\bar{C})[n]\cong J(C)(\bar{F})[n].
		\] Applying Poincaré duality \cite[\S VI Cor.~11.2]{MIL16}, which is a perfect pairing and Galois invariant
		\[
		\HH^{1}_{\et}(\bar{C},\mu_{n})\times \HH^{1}_{\et}(\bar{C},\mathbb{Z}/n\mathbb{Z})\rightarrow \HH^{2}_{\et}(\bar{C},\mu_{n})\cong \mathbb{Z}/n\mathbb{Z}\cong \tfrac{1}{n}\mathbb{Z}/\mathbb{Z}
		\] induces the following isomorphism
		\[\HH^{1}_{\et}(\bar{C},\mathbb{Z}/n\mathbb{Z}) \cong \Hom(\HH^{1}_{\et}(\bar{C},\mu_{n}),\tfrac{1}{n}\mathbb{Z}/\mathbb{Z})\cong \Hom(J(C)(\bar{F})[n],\mathbb{Q}/\mathbb{Z}).\]
		By definition there is a canonical injection $
		\Br V/\Br_{1} V \hookrightarrow \Br(\bar{V})^{\Gal(\bar{F}/F)}$, hence by Lemma \ref{lemma: Brauer group over algebaric closure} we have an inclusion
		\[
		\left(\Br V/\Br_{1} V\right)[n] \hookrightarrow \Hom(J(C)(\bar{F})[n],\mathbb{Q}/\mathbb{Z})^{{\Gal(\bar{F}/F)}}.
		\]
		In the case that $C$ has genus $1$ we can use the kernel of the corresponding Galois invariant homomorphism $\phi\colon J(C)[n](\bar F) \to \mathbb Q/\mathbb Z$, to construct the $F$-isogeny $J(C) \twoheadrightarrow J(C)/\ker \phi$. As we started from a point of order $n$, we see that $\phi$ has precisely order $n$. This implies that $\Image \phi=\frac1n\mathbb Z/\mathbb Z$ hence $\#\ker \phi = n$.
	\end{proof}
	We also need the following result on étale cohomology.
	
	\begin{lemma}\label{lemma: emedding cyclic extensions and torsors}
		Let $F$ be a field that does not contain a primitive third root of unity and $L:=F(\sqrt[3]{a})=F[x]/(x^3-a)$ for some $a\in F\setminus F^{\ast 3}$. Then for a trivial Galois module $M$ the restriction map
		\[
		\HH^{1}_{\et}(F,M)\rightarrow \HH^{1}_{\et}(L,M)
		\]
		is injective.
	\end{lemma}
	
	\begin{proof}
		Let us write $\Gamma_L \subset \Gamma_F$ for the absolute Galois groups of $L$ and $F$, respectively. An element of the domain is a homomorphism $\Gamma_F \to M$ which lies in the kernel of the restriction map if and only if it is zero on the non-normal subgroup $\Gamma_L$ of index $3$. As such $\Gamma_F \to M$ must be trivial on the smallest normal subgroup of $\Gamma_F$ which contains $\Gamma_L$, hence the kernel equals $\Gamma_F$.
	\end{proof}
	
	\subsection{Jacobian of diagonal cubic curves}
	As we saw in Proposition \ref{prop: isogenies and transcendental elements}, it is crucial to study the Jacobian of $D$. This section is dedicated to finding which cyclic $p$-isogenies $J(D)$ can have. Note that if $F$ is field of characteristic not equal to 2 or 3, then the Jacobian of $D/F$ is the curve 
	\[
	J(D)\colon \quad x_{1}^3+x_{2}^3+a_{1}a_{2}a_{3}x_{3}^3=0.
	\]
	By fixing the point $\mathcal{O}:=[1:-1:0]$ on $J(D)$, the 3-torsion $J(D)(\bar{F})[3]$ is defined by $x_{1}x_{2}x_{3}=0$ and we can rewrite the defining equation for $J(D)$ in Weierstrass form 
	\[
	J(D):y^2z=x^3-27(4a_{1}a_{2}a_{3})^2z^3.
	\]

	\begin{lemma}[{\cite[Table 5]{AG21}}]\label{lemma: isogeny classes of mordell curves}
		Let $D$ be the plane curve \[
		D:x_{1}^3+x_{2}^3+a x_{3}^3=0\subset \mathbb{P}^{2}_{\mathbb{Q}}
		\]  where $a$ is a cube-free integer. Then the $\mathbb{Q}$-isogeny classes of $D$ are shown in Table \ref{Tab:Isogeny classes of D}. 
	\end{lemma}
	
	\begin{center}
		\begin{tabular}{|p{1.4cm}|p{5cm}|p{3.0cm}|p{4.8cm}|}
			\hline
			\multicolumn{4}{|c|}{Isogeny Classes} \\
			\hline
			& Isogenous curves & Isogeny degree &Torsion of isogenous curve\\
			\hline
			$a=1$   & $D$   & 1 &   $\mathbb{Z}/3\mathbb{Z}$\\
			&  $y^2+y=x^3-270x-1708$  & 3   & Trivial\\ 
			&  $y^2+y=x^3$ & 3 &  $\mathbb{Z}/3\mathbb{Z}$\\
			&$y^2+y=x^3-30x+63$ & 9 &$\mathbb{Z}/3\mathbb{Z}$\\ \hline 
			$a =2$ &   $D$  & 1 & $\mathbb{Z}/2\mathbb{Z}$\\
			& $y^2=x^3-135x-594$  & 2   & $\mathbb{Z}/2\mathbb{Z}$\\
			& $y^2=x^3+1$  & 3& $\mathbb{Z}/6\mathbb{Z}$\\
			& $y^2=x^3-15x+22$ & 6 & $\mathbb{Z}/6\mathbb{Z}$\\
			\hline
			$a\neq 1,2$& $D$ & 1 & Trivial\\
			& $y^2=x^3+(4a)^2$ & 3 & $\mathbb{Z}/3\mathbb{Z}$\\ \hline
		\end{tabular}
		\captionof{table}{Isogeny classes of $D$\label{Tab:Isogeny classes of D}}
	\end{center}
	
	Only these $3$-isogenies will appear over a field extensions $K/\QQ$ when $K$ does not contain a primitive third root of unity.
	
	\begin{lemma}\label{lem: isogeny of Mordell curve over cubic extension}
		Let $K$ be a number field not containing a primitive third root of unity and $D$ the elliptic curve over $K$ defined by
		\[
		D\colon\quad x_{1}^{3}+x_{2}^3+ax_{3}^3=0\subseteq \mathbb{P}^2_{K}.
		\]
		\begin{enumerate}[label=\emph{(\arabic*)}]
			\item If $a\in K^{\ast 3}$ then $D$ has precisely two cyclic 3-isogenies over $K$,
			\item Otherwise, $D$ has precisely one cyclic 3-isogeny over $K$.
		\end{enumerate}
	\end{lemma}
	
	\begin{proof}
		Writing $D$ in Weierstrass form $D:y^2=x^3-27(4a)^2$ we can consider the action of $\Gal(\bar{K}/K)$ on $D[3]$. Denote by $x_{i}$ the roots of $x^3-27(4^3a^2)$ for $i\in\{1,2,3\}$. The order $3$ subgroups of $D[3]$ written in $(x,y)$ coordinates are \begin{align*}
			S_{0}&:=\{\mathcal{O},(0,12a\sqrt{-3}),(0,-12a\sqrt{-3})\},&& S_{i}:=\{\mathcal{O},(x_{i},36a),(x_{i},-36a)\}.
		\end{align*}
		The subgroup $S_{0}$ is Galois invariant so $D$ has at least one cyclic 3-isogeny over $K$. If $a\not\in K^{\ast 3}$ then $S_{1},S_{2}$ and $S_{3}$ are permuted by $\Gal(\bar{K}/K)$, hence $D\rightarrow D/S_{0}$ is the unique (up to isomorphism) cyclic 3-isogeny of $D$. If $a\in K^{\ast 3}$ then there exists $i \in \{1,2,3\}$ such that $S_{i}$ is Galois invariant. Without loss of generality we can assume $i=1$, then $S_{2}$ and $S_{3}$ will be permuted by $\Gal(\bar{K}/K)$ as $K$ does not contain a primitive third root of unity, hence the statement.
	\end{proof}
	
	\begin{proposition}\label{prop:no points of order 9}
		Both $3$-isogenous curves of $x_1^3+x_2^3+x_3^3 =0$ do not contain a $K$-point of order $9$ for any pure multicubic fields $K = \mathbb Q(\sqrt[3] b_1, \sqrt[3] b_2,\ldots)$.
	\end{proposition}
	
	Note that this result proves that over $\QQ$ no curve $3$-isogeneous to the general curve $a_1x_1^3+a_2x_2^3+a_3x_3^3 =0$ has a $9$-torsion point.
	
	\begin{proof}
		One has to check that the field of definition of any $9$-torsion point of the curve $x_1^3+x_2^3+x_3^3 =0$ does not embed into a pure multicubic field. One way is to start from the $9$th division polynomial, as the roots of the polynomial are the $x$-coordinates of the $9$-torsion points. We will only need to consider the factors of this polynomial whose degree is a power of $3$, as the field it generates should embed in $K$. For the curve $y^2+y=x^3$ we find $2$ relevant factors (out of the $7$ in total), but one can check that both fields generated by such a factor contain a cyclic cubic extension, which is not the case for $K$. For the curve $y^2+y=x^3-270x-1708$ we find $3$ factors whose degree is a power of $3$ (out of the $5$ factors in total), but in only two cases the factor generates a field with a cyclic cubic subfield. The remaining factor generated the field $\mathbb Q(\sqrt[3]3)$, and the curve does have $3$-torsion, but no $9$-torsion defined over this field.
	\end{proof}

	\subsection{Proof of Theorem \ref{thm: Brauer group thm}}
	
	The main ingredient of the proof of Theorem \ref{thm: Brauer group thm} will be the purity sequence
	\[
	0 \rightarrow \Br X \rightarrow \Br U \xrightarrow[]{\partial_{D}} \HH_{\et}^{1}(D,\mathbb{Q}/\mathbb{Z})
	\]
	and an understanding of the curves isogenous to $D=\{a_{1}x_1^3+a_2x_2^3+a_3x^3=0\}\subset \PP^2_{\QQ}$.
	
	\begin{proposition}\label{prop:no 3-torsion}
		The quotient $\Br U/\Br X$ has no $3$-torsion.
	\end{proposition}
	
	\begin{proof}
		By Proposition~\ref{prop: algebaric Brauer group for affine is the same as projective} the inclusion $U \subset X$ over $\mathbb Q$ induces an isomorphism $\Br X \to \Br_1 U$, and the same hold after base change to any number field $K$. Hence
		\[
		\Br U/\Br X \cong \Br U/\Br_1 U \hookrightarrow  \Br U_K/\Br_1 U \cong \Br U_K/\Br X_K
		\]
		and it suffices to prove the claim after a base change. We will work over a pure multicubic extension $K/\QQ$ so that we can assume that $X$ is given by the equation $x_1^3 + x_2^3 + x_3^3 = a_0x_0^3$.
		
		We will use $P_1 = [1\colon -1 \colon 0]\in D(K)$ and $P_2 = [1\colon 0 \colon -1]\in D(K)$. We follow \cite[Prop.~3.1]{CTW12} to show that the image of $\partial_D$ is contained in the kernel of
		\[
		\alpha=(\alpha_1,\alpha_2) \colon \HH_{\et}^{1}(D,\mathbb{Q}/\mathbb{Z}) \to \HH_{\et}^{1}(P_1,\mathbb{Q}/\mathbb{Z}) \oplus \HH_{\et}^{1}(P_2,\mathbb{Q}/\mathbb{Z}):
		\]
		there is a line $L_i \subset X$ which intersects $D$ in $P_i$ transversally so we have
		\[ \begin{tikzcd}
			0 \arrow{r} & \Br X \arrow{r} \arrow[swap]{d} & \Br U \arrow{d} \arrow{r}{\partial_{D}} & \HH^{1}_{\et}(D,\mathbb{Q}/\mathbb{Z}) \arrow{d}{\alpha_i} \\
			0 \arrow{r} &\Br L_i \arrow{r}& \Br (L_i\backslash P_i) \arrow{r}{\partial_{P_i}} & \HH^{1}_{\et}(P_i,\mathbb{Q}/\mathbb{Z})
		\end{tikzcd}
		\]
		As $\sA \in \Br U$ is constant on $L_i\setminus P_i \cong \AA^1_K$ as $\Br \AA^1_K=\Br K$, the class $\sA$ has trivial residue at $P_i$.
		
		Now suppose we have a degree $3$ cover of $\psi \colon E \to D$ classified by an element of the kernel of $\alpha$, then the fibres over both $P_i$ consists of three $K$-points. Choosing a $K$-point $Q_0 \in \psi^{-1} P_0$, makes $(E,Q_0) \to (D,P_0)$ into an isogeny of elliptic curves of degree $3$. For this choice $P_1$ has order $3$ on $D$, and all points in $\psi^{-1} P_1$ have order $9$ on $E$. By Proposition~\ref{prop:no points of order 9} there are no such isogenies of $D$ of degree $3$ over pure multicubic fields $K/\QQ$. So $(\Image \partial_D)[3] \subseteq (\ker \alpha)[3] =0$ and we conclude that $(\Br U/\Br X)[3]=0$.
	\end{proof}
	
	Having dealt with the 3-torsion in the transcendental Brauer group we now deal with the 2-torsion.
	
	\begin{proposition}\label{prop:2-torsion}
		If $a_1a_2a_3 \equiv 2 \mod \mathbb Q^{*3}$ then $(\Br U/\Br X)[2] = \ZZ/2\ZZ$.
	\end{proposition}
	
	\begin{proof}
		Consider the purity sequence
		\[
		0 \rightarrow \Br X \rightarrow \Br U \xrightarrow[]{\partial_{D}} \HH_{\et}^{1}(D,\mathbb{Q}/\mathbb{Z}) \xrightarrow{\theta} \HH_{\et}^{1}(X,\mathbb{G}_m).
		\]
		Let $L\subset \mathbb{P}^{2}_{F}$ be a line not tangent to $D$, which can be chosen by Bertini's Theorem \cite[Thm.~6.3(4)]{J83}. By Bezout's Theorem $D$ and $L$ intersect at 3 distinct points geometrically. Consider for $i=1,2$ the maps
		\[
		\sigma_{i}: \HH^{i}_{\et}(D,\QQ/\ZZ)\rightarrow \HH^{i}_{\et}(\QQ,\QQ/\ZZ), \quad m\mapsto \sum\limits_{P\in L\cap D}\Cores_{\QQ(P)/\QQ}m(P).
		\]
		It is shown \cite[Lem.~5.4]{CTW12} that $m \in \HH^{i}_{\et}(D,\QQ/\ZZ)$ of an order coprime to $3$ lies in $\ker \theta$ if and only if it lies in $\ker \sigma_1$. Hence upon passing to the $2$-torsion
		\[
		(\Br X/\Br U)[2] \xrightarrow[\partial_D]{\cong} (\ker \theta)[2] = (\ker \sigma_1)[2] =\ker\left(\HH^{i}_{\et}(D,\tfrac12\ZZ/\ZZ) \to \HH^{i}_{\et}(\QQ,\tfrac12\ZZ/\ZZ) \right)
		\]
		as the $2$-torsion of cohomology with coefficient $\QQ/\ZZ$ equals the cohomology with coefficients $\tfrac12 \ZZ/\ZZ$.
		
		Consider the 5-term exact sequence coming from the Hochschild--Serre spectral sequence 
		\[
		\begin{tikzcd}
			0 \arrow[r] & \HH^1_{\et}(\QQ,\tfrac12 \ZZ/\ZZ) \arrow[r,"\xi_{1}"] & \HH^1_{\et}(D,\tfrac12 \ZZ/\ZZ) \ar[draw=none]{d}[name=X, anchor=center]{}\arrow[r]
			& \HH^{1}_{\et}(\bar{D},\tfrac12 \ZZ/\ZZ)^{\Gal(\bar{\QQ}/\QQ)} \ar[rounded corners,
			to path={ -- ([xshift=2ex]\tikztostart.east)
				|- (X.center) \tikztonodes
				-| ([xshift=-2ex]\tikztotarget.west)
				-- (\tikztotarget)}]{dll}[at end]{}       &               &    \\
			& \HH^{2}_{\et}(\QQ,\tfrac12 \ZZ/\ZZ) \arrow[r,"\xi_{2}"] & \HH^{2}_{\et}(D,\tfrac12 \ZZ/\ZZ).  {}
		\end{tikzcd}
		\]
		We will show that $\sigma_i$ is a one-sided inverse for $\xi_i$. The composition $\xi_i\circ \sigma_i$
		\[
		\HH^{i}_{\et}(\QQ,\tfrac12 \ZZ/\ZZ)\xrightarrow[]{\xi_{i}} \HH^{i}_{\et}(D,\tfrac12 \ZZ/\ZZ)\xrightarrow[]{\sigma_{i}} \HH^{i}_{\et}(\QQ,\tfrac12 \ZZ/\ZZ)
		\]
		equals multiplication by 3 on an elementary $2$-group, hence it is the identity. So
		\begin{equation*}
			\ker\left(\HH^{1}_{\et}(D,\tfrac12 \ZZ/\ZZ)\xrightarrow[]{\sigma_{1}} \HH^{1}_{\et}(\mathbb{Q},\tfrac12 \ZZ/\ZZ)\right) =\HH^{1}_{\et}(\bar{D},\tfrac12 \ZZ/\ZZ)^{\Gal(\bar{\mathbb{Q}}/\mathbb{Q})}.
		\end{equation*}
		As in the proof of Proposition~\ref{prop: isogenies and transcendental elements} the last group classifies isogenies over $J(D)$ of degree at most $2$. From Table~\ref{Tab:Isogeny classes of D} we see these form the group $\ZZ/2\ZZ$.
	\end{proof}
	
	We are now set to prove Theorem \ref{thm: Brauer group thm}.
	
	\begin{proof}[Proof of Theorem \ref{thm: Brauer group thm}]
		From Proposition \ref{prop: isogenies and transcendental elements} and Lemma \ref{lemma: isogeny classes of mordell curves}, we see that the only possible torsion that appears in $\Br U/\Br X$ is $1,2,3$ and $6$, and the even torsion only appears if $a_1a_2a_3 \equiv 2 \mod \QQ^{*3}$. Now proposition~\ref{prop:no 3-torsion} states there is no $3$-torsion. We conclude the proof by using Proposition~\ref{prop:2-torsion} $\Br U/\Br X = \ZZ/2\ZZ$ in the case $a_1a_2a_3 \equiv 2 \mod \QQ^{*3}$.
	\end{proof}
	
	\subsection{Explicit generators}\label{sec:2-coeff} For future computations it will useful to write down explicit elements for the Brauer group of diagonal affine cubic surfaces. As we have seen in the generic case this Brauer group is isomorphic to the Brauer group of its compactification which was studied by Colliot-Thélène, Sansuc and Kanevsky \cite{CTKS87}. Unfortunately, there is no uniform generator for the family of all diagonal cubic surfaces \cite[Thm.~1.2]{U14}. However, in specific subfamilies one is able to write down generators. In particular, this is possible if $a_1 = a_2 = 1$, which we will assume throughout this subsection. Colliot-Thélène and Wittenberg have made an extensive study in \cite{CTW12} of the case where $a_3 = 1$ or $a_3 = 2$. Note that the compactification $X$ of $U$ always has a non-empty set of rational points, namely $[0:1:-1:0] \in X(\QQ)$.
	
	\begin{lemma}[{\cite[\S1, Prop.~1]{CTKS87}}]
		\label{lem:BrX}
		If $a_1 = a_2 = 1$ and $a_0, a_3$ are cube-free, then
		\[ 
		\Br X/\Br \mathbb{Q} =
		\begin{cases} 
			0 \ &\text{if}\ a_{0}a_{3} \ \text {or}\ a_{0}/a_{3}\ \text{are in}\ \mathbb{Q}^{\ast 3},\\
			\mathbb{Z}/3\mathbb{Z}\ &\text{otherwise}.
		\end{cases}
		\] 
		In particular, if $\Br X/\Br \mathbb{Q}\cong \mathbb{Z}/3\mathbb{Z}$, then $\Br X/\Br \mathbb{Q}$ is generated by the cyclic algebra 
		\[
		\sB' = \Cores_{\mathbb{Q}(\omega)/\mathbb{Q}} \sB \in \Br \mathbb{Q}(X), \text{ where} \quad 
		\sB = \left(\frac{a_{0}}{a_{3}},\frac{x_{1}+\omega x_{2}}{x_{1}+x_{2}}\right)_{\omega}.
		\]
	\end{lemma}
	\begin{lemma}
		\label{lem:BrU}
		Assume that $a_1 = a_2 = 1$ and $a_0, a_3$ are cube-free. Then
		\begin{equation*} 
			\Br U=
			\begin{cases} 
				\Br X \oplus \mathbb{Z}/2\mathbb{Z} \ &\text{if}\ a_{3} \equiv 2\ \text{mod}\ \mathbb{Q}^{\ast 3},\\
				\Br X &\text{otherwise}.
			\end{cases}
		\end{equation*} 
		In particular, if $a_{3} \equiv 2\ \text{mod}\ \mathbb{Q}^{\ast 3}$, then the $\mathbb{Z}/2\mathbb{Z}$ factor is generated by the transcendental element
		\[
		\left(a_{0}(x_{1}+x_{2}+2x_{3}),-3(x_{1}+x_{2}+2x_{3})(x_{1}+x_{2})\right)\in \Br \mathbb{Q}(U).
		\]
	\end{lemma}
	\begin{proof}
		The main statement is proved in Theorem \ref{thm: Brauer group thm} and the explicit generator for the $\mathbb{Z}/2\mathbb{Z}$ factor was determined in \cite[Prop.~3.4]{CTW12}.
	\end{proof}
	\begin{remark} Consider the case
		$a_{1}=1$ and
		$a_{2}a_{3}\equiv 2$ mod $\mathbb{Q}^{* 3}$. Over the field extension
		$K:=\mathbb{Q}[x]/(x^3-a_{2})$ the surface $U$ becomes isomorphic to the surface
		\[
		U':\quad u_{1}^3+u_{2}^3+2u_{3}^3=a_{0}.
		\]
		In the absence of primitive third roots of unity in the base field, one can show that the curve $u_{1}^3+u_{2}^3+2u_{3}^3=0$ has a unique degree $2$ isogeny, and as in the proof of Proposition~\ref{prop:2-torsion} we can deduce $(\Br U'/\Br_{1}U')\cong \mathbb{Z}/2\mathbb{Z}$. This can be used to show that $\Br U$ is generated over $\Br_1 U$ by the corestriction from $K$ to $\mathbb{Q}$ of the $2$-torsion element in Lemma \ref{lem:BrU}.
	\end{remark}
	
	\section{Hilbert symbols}\label{sec: computing inv}
	This section describes how to compute the local invariant map. We begin by defining Hilbert symbols and describe their relation to the invariant map. From there we give an algorithm from \cite{CTKS87} which will enable us to compute the Brauer--Manin set for generic diagonal affine cubic surfaces. Here $U$ is as defined in \eqref{eqn:U-main} and we keep the notational convention set up earlier.
	
	\subsection{Construction of the Hilbert symbol}
	Let $K$ be a number field and $v$ a place of $K$. Assume $K_{v}$ contains a primitive $n$th root of unity $\omega_{n}$. There exists a pairing \cite[Chap XIV, \S 2]{S79}
	\[
	(\cdot,\cdot)_{\omega_{n}}:K_{v}^{\ast}/K_{v}^{\ast n}\times K_{v}^{\ast}/K_{v}^{\ast n} \rightarrow \Br K_{v}, \ (a,b)\mapsto (a,b)_{\omega_{n}}.
	\] In $\Br K_{v}$ we have the relations
	\begin{equation}
		\label{eqn:Hilb-general}
		\begin{split}
			(aa', b)_{\omega_{n}} = (a, b)_{\omega_{n}} + (a', b)_{\omega_{n}}, \quad
			(a, b)_{\omega_{n}} = -(b, a)_{\omega_{n}}, \quad
			(a^n , b)_{\omega_{n}} = 0.
		\end{split}
	\end{equation}
	
	We define the \emph{Hilbert symbol}
	\[
	(a, b)_{\omega_{n},v} = \inv_v (a, b)_{\omega_{n}} \in \QQ/\ZZ.
	\]
	For any non-zero prime ideal $\mathfrak{p}$ of $K$, which does not lie over any of the prime divisors of $n$, with associated place $v$
	\begin{equation}
		\label{eqn:Hilb-units-uniformiser}
		(u, u')_{\omega_{n}, v} = 0, \quad \text{and} \quad 
		(u, \pi)_{\omega_{n}, v} = 0 \Leftrightarrow u \in K_{v}^{n} ,
	\end{equation}
	for $u, u' \in \mathcal{O}_{K_v}^{\ast}$ and $\pi \in \mathcal{O}_{K_{v}}$ a uniformiser. For $n=3$ we recall some formulae from \cite[\S4]{CTKS87} who identify $\tfrac{1}{3}\mathbb{Z}/\mathbb{Z}$ with $\mathbb{Z}/3\mathbb{Z}$, which agrees with \cite[Chap XIV, \S 2]{S79}. Throughout the rest of the paper we make the same identification. 
	
	Let $\omega$ be a fixed root of $x^2 + x + 1$. We fix an isomorphism $k:=\mathbb{Q}(\omega) \cong \mathbb{Q}[x]/(x^2+x+1)$ by sending $x$ to $\omega$.
	
	\begin{enumerate}
		\item Let $p$ be a prime where $p\equiv 1$ mod 3. As ideals in $\mathcal{O}_{k}$ we have $(p)=\mathfrak{p}_{1}\mathfrak{p}_{2}$. Denote by $v$ a place corresponding to the prime ideal $\mathfrak{p}_{1}$ or $\mathfrak{p}_{2}$ extending from the $p$-adic valuation on $\mathbb{Q}$. If $u$ is a unit in the ring of integers of $\mathbb{Q}_{p}\cong k_{v}$, then we have the formula 
		\begin{equation}
			\label{eqn:Hilbert-unit-uniformiser-p=1(3)}
			(u,p)_{\omega,v}=-i \text{ mod }3
		\end{equation} 
		where $u^{\frac{p-1}{3}}\equiv \omega^{i}$ mod $p$. In particular $(\omega,p)_{v}\equiv -\frac{p-1}{3}$ mod $3$.
		
		\item Let $p$ be a prime where $p\equiv 2$ mod 3. As ideals in $\mathcal{O}_{k}$ we have $(p)=\mathfrak{p}$. Denote by $v$ a place corresponding to the prime ideal $\mathfrak{p}$ extending from the $p$-adic valuation on $\mathbb{Q}$. If $u$ is a unit in the ring of integers of $\mathbb{Q}_{p}(\omega)\cong k_{v}$, we have the formula 
		\begin{equation}
			\label{eqn:Hilbert-unit-uniformiser-p=2(3)}
			(u,p)_{\omega,v}=-i \text{ mod }3
		\end{equation}
		where $u^{\frac{p^{2}-1}{3}}\equiv \omega^i$ mod $p$. In particular we have $(\omega,p)_{\omega, v}\equiv -\frac{p^2-1}{3}$ mod $3$.
		
		\item Let $p=3$, then as ideals in $\mathcal{O}_{k}$ we have $(3)=\mathfrak{p}^2$. Note $\mathcal O_{k} = \mathbb Z[\omega]$ and the prime ideal $\mathfrak{p}$ is generated by $\lambda'= 2\omega+1$ which satisfies $\lambda'^2=-3$. We choose the uniformizer $\lambda =  \lambda'^2\omega +\lambda' = -3\omega + (2\omega + 1) = 1-\omega $, with minimal polynomial $\lambda^2-3\lambda+3=0$ and the following relations 
		\begin{equation}\label{eqn: equations for lambda at the prime 3}
			\begin{aligned}
				\omega &= 1-\lambda, &&3\equiv -\lambda^2-\lambda^3 \text{ mod } \lambda^4,\\
				3 &= -\omega^2\lambda^2, &&2\equiv -1-\lambda^2-\lambda^3 \text{ mod } \lambda^4.
			\end{aligned}
		\end{equation}
		The relations for $(-,-)_{\mathfrak p}$ are given in \cite[p. 34]{CTKS87}, however we give a summary of the most important facts. Any element in $\mathbb Z_3[\omega]$ can be written as $\pm \lambda^eu$ where $u$ is a $1$-unit.  We will write a $1$-unit $u_b$ as $1+b_1\lambda+b_2\lambda^2+\ldots$ with $b_i \in \mathbb Z$. To compute the symbol at $\mathfrak p$ with associated place $v$, we will only need the following information
		\begin{align}
			(u_b, u_c)_ {\omega,v} &\equiv b_1c_1(b_1-c_1) -b_1c_2+b_2c_1 \bmod 3,\label{eqn: hilbert symbol at 3 for two units}\\
			(\lambda,u_b)_{\omega, v} &\equiv \frac{b_1-b_1^3}3 + b_1b_2-b_3 \bmod 3.\label{eqn: hilbert symbol at 3 for uniformiser and unit}
		\end{align}
	\end{enumerate}
	
	\begin{remark}
		Note the relation (\ref{eqn: hilbert symbol at 3 for two units}) only depends on $b_i,c_i$ modulo $3$ and not on any $b_i,c_i$ with $i\geq 3$. For the relation (\ref{eqn: hilbert symbol at 3 for uniformiser and unit}) we need to know $b_1$ modulo $9$, $b_2,b_3$ modulo $3$ and none of the $b_i$ with $i \geq 4$.
	\end{remark}
	
	\subsection{Conditions for a Brauer--Manin obstruction}
	We are now in position to establish sufficient conditions for the integral Brauer--Manin set of $\sU$ to be non-empty. This is done in Propositions \ref{prop:p divides a_0} and \ref{prop:p divides a_1 but}. We write $\mathcal{B}'$ for both the preimage in $\Br X$ of a generator of $\Br X/\Br_{0}X$ and its restriction to $U$ when $\Br X/\Br_{0}X$ is cyclic. 
	
	\begin{proposition}
		\label{prop:p divides a_0}
		Assume that there exists a prime $p\neq 3$ such that $p \mid a_{0}$ but $p^{3}\nmid a_{0}$ and $p \nmid a_1 a_2 a_3$. If $\mathcal{U}(\mathbb{Z}_{p})\neq \emptyset$, then $\inv_{p}\mathcal{B}' : \mathcal{U}(\mathbb{Z}_{p})\rightarrow \frac{1}{3}\mathbb{Z}/\mathbb{Z}$ is surjective.
		
	\end{proposition}
	
	\begin{proof}
		Under the conditions of the proposition we have $\Br X/\Br_{0}X\cong \mathbb{Z}/3\mathbb{Z}$. Consider the composition 
		\[
		\mathcal{U}(\mathbb{Z}_{p})\xrightarrow{}{} \mathcal{X}(\mathbb{Z}_{p})\xrightarrow{\text{red}}{} E(\mathbb{F}_{p}),
		\] where $E$ is the elliptic curve \[
		a_{1}x_{1}^3+a_{2}x_{2}^3+a_{3}x_{3}^3=0.
		\] Clearly, this composition is surjective. The invariant map $\inv_{p}\mathcal{B}'$ factors as the surjective map $\text{red}$ and a surjective homomorphism $E(\mathbb{F}_{p})\rightarrow \frac{1}{3}\mathbb{Z}/\mathbb{Z}$, \cite[Chap~4, Thm~6.4~c)i)]{J14}. By functionality we can deduce that $\inv_{p}\mathcal{B}':\mathcal{U}(\mathbb{Z}_{p})\rightarrow  \frac{1}{3}\mathbb{Z}/\mathbb{Z}$ is surjective.
	\end{proof}
	
	\begin{proposition}
		\label{prop:p divides a_1 but}
		Assume that there are $p \ge 17$ and $i \in \{1, 2, 3\}$ such that $p \mid a_i$ but $p^3 \nmid a_i$ and $p \nmid  a_j$ for any $j \in \{0, 1, 2, 3\} \setminus \{i\}$. Then  $\inv_p \sB' : \sU(\Zp) \rightarrow \frac{1}{3}\ZZ/\ZZ$ is surjective.
	\end{proposition}
	
	\begin{proof}
		The proof is very similar to the one of Proposition~\ref{prop:p divides a_0} with the only difference that the local invariant map now factors through the composition
		\[
		\sU(\Zp) \xrightarrow{}{} \sX(\Zp) \xrightarrow{\text{red}}{} (E \setminus R)(\Fp),
		\]
		where if $i = 1$ the elliptic curve is $E: a_2x_2^3 + a_3x_3^3 - a_0x_0^3 = 0$ and $R$ is the divisor on $E$ given by the vanishing locus of $x_0$. As before this composition is surjective and hence the proof boils down to establishing that the homomorphism $(E \setminus R)(\Fp)\rightarrow \frac{1}{3}\ZZ/\ZZ$ is surjective. This follows from the Hasse--Weil bound and \cite[Chap~4, Thm~6.4~c)i)]{J14} as there are at least $(p + 1 - 2\sqrt{p})/3 - 3$ points with a given value of the local invariant map. This number is clearly positive provided that $p \ge 17$, which confirms our claim.
	\end{proof}
	\subsection{Computing invariant maps for generic families}\label{subsec: invariant map for generic families}
	We give an overview of the work in \cite{CTKS87} which describes how to compute the Brauer--Manin set for diagonal cubic surfaces. Assume that $a_{i}\in \mathbb{Z}_{\ne 0}$ are cube-free. Dividing the defining equation of $X$ by $a_{0}$ gives
	\[
	X\colon \quad x_{0}^3+\lambda x_{1}^3+ \mu x_{2}^3 + \lambda \mu \nu x_{3}^3 = 0,
	\]
	where $\lambda, \mu, \nu \in \mathbb Q^{\ast}$ are $\lambda = a_1/a_0$, $\mu = a_2/a_0$ and $\nu = -a_3a_0/(a_1a_2)$. Pick $\alpha, \gamma, \omega \in \bar{\mathbb Q}$, such that $\alpha^3=\lambda$, $\gamma^3 = \nu$, $\omega^2+\omega+1=0$.
	Define $K:=\mathbb Q(\omega,\alpha, \gamma)$ which is Galois over $k:= \mathbb Q(\omega)$, with Galois group $G$ which is isomorphic to $(\mathbb Z/3\mathbb Z)^2$ and generated by $s,t$ defined by
	\[
	{ }^s\alpha = \alpha, \quad { }^t\alpha = \omega\alpha, \quad { }^s\gamma  = \omega\gamma, \quad { }^t\gamma  = \gamma.
	\]
	Let $\beta:=\alpha\gamma$ and $\delta:=\alpha/\gamma$. This defines the following Galois extensions
	\begin{align*}
		\alpha & = \sqrt[3]{\lambda} & \langle s\rangle = & \Gal(K/K_1)\\
		\beta & = \sqrt[3]{\lambda\nu} = \alpha \gamma & \langle q\rangle = & \Gal(K/K_2), q=st^2 \text{ so } { }^q\beta=\beta \text{ and } { }^q\delta=\omega^2\delta\\
		\gamma & = \sqrt[3]{\nu} & \langle t\rangle = & \Gal(K/L_1)\\
		\delta & = \sqrt[3]{\lambda/\nu} = \alpha/\gamma & \langle r\rangle = &\Gal(K/L_2), r=st \text{ so } { }^r\beta=\omega^2\beta \text{ and } { }^r\delta=\delta.
	\end{align*}
	
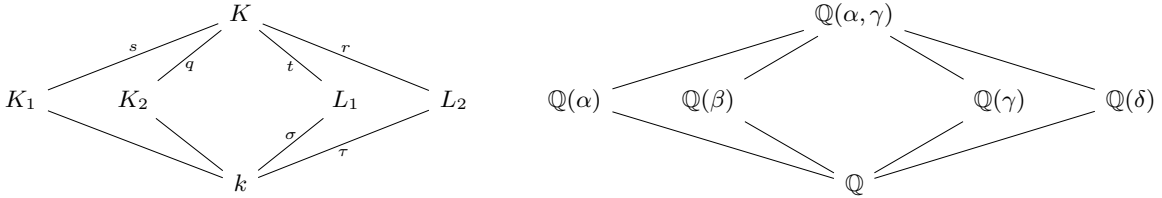
	\begin{center}
		{\scriptsize
			\begin{tikzcd}
				& & K \arrow[dll, dash, "s" above] \arrow[dl, dash, "q" below] \arrow[dr, dash, "t" below] \arrow[drr, dash, "r" above] & & & & & \mathbb Q(\alpha,\gamma) \arrow[dll, dash] \arrow[dl, dash] \arrow[dr, dash] \arrow[drr, dash] & &\\
				K_1 \arrow[drr, dash] & K_2 \arrow[dr, dash] & & L_1 \arrow[dl, dash, "\sigma" above] & L_2 \arrow[dll, dash, "\tau" below] & \mathbb Q(\alpha) \arrow[drr, dash] & \mathbb Q(\beta) \arrow[dr, dash] & & \mathbb Q(\gamma) \arrow[dl, dash] & \mathbb Q(\delta) \arrow[dll, dash]\\
				& & k & & & & & \mathbb Q & &
		\end{tikzcd}}
	\end{center}
	\begin{remark}
		\label{rem:constant-inv}
		Let $S$ be a smooth, projective variety over a number field $K$ and $v$ a place of $K$. Then if $S$ is $K_{v}$-rational we have $\Br S_{K_{v}}=\Br K_{v}$ i.e. the invariant map at $v$ is constant for any element in $\Br S$. In the case of diagonal cubic surfaces $X/K$, Colliot-Thélène, Kanevsky and Sansuc give a necessary and sufficient condition for $X$ to be $K_{v}$-rational \cite[\S 5, Lem.~8]{CTKS87}. Namely, let $F$ be a field of characteristic not equal to 3, then $X/F$ is $F$-rational if and only if $X(F)\neq \emptyset$ and $a_{0}a_{1}/a_{2}a_{3}$ is a cube in $F^{\ast}$.
	\end{remark}

	In the case $\Br X/\Br_{0} X\cong \mathbb{Z}/3\mathbb{Z}$, as in \cite[\S3]{CTKS87} we choose a generator $\mathcal{B}\in \Br X_{k}$ such that $\mathcal{B}':=\Cores_{k/\mathbb{Q}}\mathcal{B}$ generates $\Br X/\Br_{0}X$. Then for a place $v$ \[
	\inv_{v}\mathcal{B}'=\sum\limits_{w\mid v}\inv_{w}\mathcal{B}.
	\]
	
	Let $v$ be a finite place of $\mathbb{Q}$ and $w$ a place of $k$ above $v$. Moreover, let $w'$ be a place of $K$ lying above $w$. Table \ref{table:inv maps} describes how to compute $\inv_w \mathcal{B}(P_v)$ at a point $P_v \in X(\mathbb Q_v) \subseteq X(k_w)$ which is dependent on the decomposition group $G^{v}=\Gal(K_{w'}/k_{w})$.

	\begin{center}
		\begin{tabular}{c|c|c|c}
			Condition on $a_{0}$, $a_{1}$, $a_{2}$, $a_{3}$ & Condition on $\lambda,\nu$ & $G^v$ & $\inv_w\mathcal B(P_v)$\\
			\hline\hline
			$a_{0}/a_{1},a_{2}/a_{3}\in \mathbb Q_{v}^{\ast 3}$, or & $\lambda, \nu \in \mathbb Q_{v}^{\ast 3}$ & $\langle e\rangle$ & $= 0$\\
			$a_{0}a_{3}/a_{1}a_{2}, a_{0}a_{2}/a_{1}a_{3}\in \mathbb Q_{v}^{\ast 3}$ & & &\\
			\hline
			$a_{0}a_{3}/a_{1}a_{2}\in \mathbb Q_{v}^{\ast 3}$ & $\nu\in \mathbb Q_{v}^{\ast 3}$ & $\langle t\rangle$ & $= 0$\\
			\hline
			$a_{0}a_{2}/a_{1}a_{3}\in \mathbb Q_{v}^{\ast 3}$ & $\lambda/\nu\in \mathbb Q_{v}^{\ast 3}$ & $\langle r\rangle$ & $({ }^q\epsilon/\epsilon \cdot \eta,\lambda)_{\omega,w} = ({ }^q\epsilon/\epsilon \cdot \eta,\nu)_{\omega,w}$\\
			\hline
			$a_{1}/a_{0}\in \mathbb Q_{v}^{\ast 3}$ & $\lambda\in \mathbb Q_v^{\ast 3}$ & $\langle s\rangle$ & $(f(P_v)/\eta,\nu)_{\omega,w}$ \\
			\hline
			$a_{3}/a_{2}\in \mathbb Q_{v}^{\ast 3}$ & $\lambda\nu\in \mathbb Q_{v}^{\ast 3}$ & $\langle q\rangle$ & $({ }^r\epsilon/\epsilon \cdot 1/h(P_v)\cdot { }^t\eta ,\lambda)_{\omega,w}$ \\
			\hline
			Otherwise & Otherwise & $G$ & $(N_s(\xi(P_v)) f(P_v)/\eta,\nu)_{\omega,w}$
		\end{tabular}
		\captionof{table}{Computing invariant maps of $X$}\label{table:inv maps}
	\end{center}
	Here
	\begin{align*}
		f = \frac{x_{0}+\alpha \omega x_{1}}{x_{0}+\alpha \omega^2 x_{1}} \text{  and   } h=\frac{x_{2}+\beta \omega x_{3}}{x_{2}+\beta x_{3}}.
	\end{align*} $\epsilon \in K^{\ast}, \eta\in K_{1}^{\ast}$ satisfy the following equations
	\begin{equation}\label{eq:norm equation for epsilon}
		\Norm_{K/L_{1}}(\epsilon)=-\mu \text{ and } \eta/{ }^r\eta=-\mu/\Norm_{K/K_{1}}(\epsilon).
	\end{equation} Furthermore, $\xi(P_{v})\in K_{w'}^{*}$ \cite[p. 39]{CTKS87} satisfies \[
	(1-t)(\xi(P_{v}))=g(P_{v})/\epsilon.
	\]
	\begin{remark}\label{rem:good epsilon}
		Suppose we choose $\epsilon$ such that $\epsilon = \epsilon_{\beta}\epsilon_{\delta}$ where $\epsilon_{\beta}\in K_{2}^{\ast}$ and $\epsilon_{\delta}\in L_{2}^{\ast}$ then we can choose $\eta=1$ \cite[p. 30]{CTKS87}. If $\epsilon=\epsilon'$ or $\epsilon = 1/\epsilon''$ where $\epsilon'$ and $\epsilon''$ are products of integral elements of $\mathbb{Q}(\beta)$ and $\mathbb{Q}(\delta)$ then $\inv_{w}\mathcal{A} = 0$ for all places $w$ of good reduction on $X_{k}$ \cite[p. 31]{CTKS87}.

		The last situation is clearly satisfied if $\mu \in \Norm_{\QQ(\beta)/\QQ}\QQ(\beta)^{\ast}$. Equivalently one can show that $\mu \in \Norm_{k(\beta)_{w'}/k_w}k(\beta)_{w'}^{\ast}$ for all $w \in \Omega_k$ and any place $w'$ of $k(\beta)$ lying above $w$. This local condition is satisfied in the following cases
		\begin{enumerate}
			\item[(i)] $w$ is a place of good reduction for $X_{k}$;
			\item[(ii)] $\nu$ is a cube in $k_w^{\ast}$;
			\item[(iii)] $\lambda/\nu$ is a cube in $k_w^{\ast}$, or
			\item[(iv)] $\mu/\nu$ is a cube in $k_w^{\ast}$, but $w\ne w_3$ where $w_3$ is the unique place of $k$ dividing $3$.
		\end{enumerate}
		All these statements can be found in Proposition~4 in \cite{CTKS87}.
	\end{remark}

	\section{A three coefficient family}
	\label{sec:3-coeff}
	We shall focus in this section on the affine diagonal cubic surfaces given in \eqref{eqn:U-main} with $a_1 = a_2$. It is convenient for the remainder of the section to set $k = \QQ(\omega)$. The data collected here will be used to construct an explicit family of Hasse failures in section \ref{sec:lower}, allowing us to prove the lower bound in Theorem~\ref{thm:N(B)}. As $U$ is isomorphic over $\QQ$ to the surface
	\[
	U\colon \quad u_1^3 + u_2^3 + (a_3/a_1)u_3^3 = a_0/a_1,
	\]
	we can use the results from section \ref{sec:2-coeff} for the Brauer group of $U$. Since we are no longer concerned with $\Br X$, we will abuse notation by using $\sB$ and $\sB'$ for their images under the natural map $\Br X \rightarrow \Br U$. This should cause no confusion in computing local invariant maps by functoriality. We will be primarily interested in the algebraic Brauer element from Lemma~\ref{lem:BrX}, that is
	\[
	\sB' = \Cores_{k/\QQ} \sB \in \Br U, \quad \text{where} \quad 
	\sB = \left(\frac{a_{0}}{a_{3}},\frac{x_{1}+\omega x_{2}}{x_{1}+x_{2}}\right)_{\omega} \in \Br U_{k},
	\]
	which generates $\Br U / \Br \QQ$ by Lemmas~\ref{lem:BrX} and \ref{lem:BrU} unless $a_3/a_1 \equiv 2 \bmod \QQ^{*3}$. We need the following lemma in order to evaluate the local invariant map of $\sB$.
	
	\begin{lemma}
		\label{lem:invp}
		Let $p \in \ZZ$ be a prime and fix $\frakp \mid p$ in $\mathcal{O}_k$. If $\sigma$ generates $\Gal(k/\QQ)$, then
		\[
		\inv_p \sB' = 
		\begin{cases}
			\inv_{\frakp}(1 + \sigma)\sB &\text{if } p \equiv 1 \bmod 3,\\
			\inv_{\frakp}\sB &\text{if } p = 3 \text{ or } p \equiv 2 \bmod 3.
		\end{cases}
		\]
		Moreover, $(1 + \sigma)\sB$ can be explicitly expressed as
		\[
		(1 + \sigma)\sB = \(\frac{a_0}{a_3},\frac{u_1 + \omega u_2}{u_1 + \omega^2 u_2}\)_{\omega}.
		\]
	\end{lemma}
	
	\begin{proof}
		If $p = 3$ or $p \equiv 2 \bmod 3$ there is a single prime $\frakp$ above $p$ in $k$ and thus the claim follows from \cite[Lem.~5.i]{BS16}. If $p \equiv 1 \bmod 3$, then $(p) = \frakp \sigma(\frakp)$ is split in $\mathcal{O}_k$. It follows from \cite[Lem.~5.i]{BS16} that $\inv_p\sB' = \inv_{\frakp} \sB + \inv_{\sigma(\frakp)} \sB$. It thus suffices to show that $\inv_{\sigma(\frakp)} \sB = \inv_{\frakp} \sigma(\sB)$ since the local invariant map is a homomorphism.
		
		To see this, we apply a similar analysis to the one appearing in \cite[Lem.~4.2.]{GLN22}. Consider the diagram
		\begin{align*}
			\xymatrixcolsep{4pc}\
			\xymatrix{
				\Br U_k \ar[d]^{x_p} \ar[rr]^{\sigma} && \Br U_k \ar[d]^{x_p} \\\
				\Br k_{\frakp} \ar[d]^{\inv_{\frakp}} \ar[r]_{\simeq} & \Br \Qp \ar[d]^{\inv_p} & \Br k_{\sigma(\frakp)} \ar[d]^{\inv_{\sigma(\frakp)}} \ar[l]^{\simeq} \\\
				\QQ/\ZZ \ar[r]_{\text{id}} & \mathbb{Q}/\mathbb{Z}   & \ar[l]^{\text{id}} \QQ/\ZZ
			}
		\end{align*}
		As explained in the proof of \cite[Lem.~4.2.]{GLN22} it commutes. The commutativity of the top square follows from the definition of the embeddings $k \to k_{\frakp}$ and $k \to k_{\sigma(\frakp)}$, and that of the bottom squares by \cite[Prop.~II.1.4]{Neu13}. As a conclusion we see that $\inv_{\frakp} \sB = \inv_{\sigma(\frakp)} \sigma(\sB)$ is now implied by chasing the above diagram. 
		
		We have so far shown that $\inv_p \sB' = \inv_{\sigma(\frakp)} (1 + \sigma)\sB$. Since $\sigma$ generates $\Gal(k/\QQ)$, we have $\sigma^2 = 1$ and hence $(1 + \sigma)\sB$ is Galois invariant. Thus $\inv_{\sigma(\frakp)} (1 + \sigma)(\sB) = \inv_{\sigma^2(\frakp)} \sigma(1 + \sigma)(\sB) = \inv_{\frakp} (1 + \sigma)(\sB)$ by the above commutative diagram, which confirms the first part of the statement. Finally, observe that $(1 + \sigma)\sB$ is given by
		\[
		\begin{split}
			(1 + \sigma)(\sB)
			&= \(\frac{a_0}{a_3},\frac{u_1+\omega u_2}{u_1 + u_2}\)_{\omega} + \(\frac{a_0}{a_3},\frac{u_1+\omega^2 u_2}{u_1 + u_2}\)_{\omega^2} \\
			&= \(\frac{a_0}{a_3},\frac{u_1+\omega u_2}{u_1 + u_2}\)_{\omega} + \(\frac{a_0}{a_3},\frac{u_1 + u_2}{u_1 + \omega^2 u_2}\)_{\omega}
			= \(\frac{a_0}{a_3},\frac{u_1 + \omega u_2}{u_1 + \omega^2 u_2}\)_{\omega}.
		\end{split}
		\]
		This completes the proof of Lemma~\ref{lem:invp}.
	\end{proof}
	
	Assume now that $\sU(\Zp) \neq \emptyset$. We proceed by computing the local invariant map of $\sB'$ in various cases depending on $p$ and the coefficients of $\sU$. 
	
	\begin{lemma}
		\label{lem:infinity}
		We have $\inv_{\infty} \sB' = 0$.
	\end{lemma}
	
	\begin{proof}
		The claim follows from Remark~\ref{rem:constant-inv} and the fact that $a_0/a_3 \in \RR^{*3}$.
	\end{proof}
	
	\begin{lemma}
		\label{lem:cube}
		If $p \neq 3$ and $a_0/a_3 \in \Qp^{*3}$, then $\inv_p \sB' = 0$.
	\end{lemma}
	
	\begin{proof}
		This is a consequence of \eqref{eqn:Hilb-general} and Lemma~\ref{lem:invp}.
	\end{proof}
	
	\begin{lemma}
		\label{lem:p-not-dividing-a_0a_3}
		Assume that $p \neq 3$ and $p \nmid a_0a_1a_3$. Then $\inv_p \sB' = 0$. 
	\end{lemma}
	
	\begin{proof}	
		The condition on $a_0, a_1, a_3$ implies that $a_0/a_3$ is a unit of $\Zp$. If $a_0/a_3 \in \Zp^{*3}$, then Lemma~\ref{lem:cube} implies that the local invariant map vanishes. Note that this assumption holds to any $p \equiv 2 \bmod 3$. Assume now that $p \equiv 1 \bmod 3$ and $a_0/a_3 \in \Zp^{*} \setminus \Zp^{*3}$. Since $p \nmid a_0a_3$ we have $a_0 - a_3u^3 \neq 0 \bmod p$ and hence none of $u_1 + \omega u_2$ and $u_1 + \omega^2 u_2$ vanishes mod $\frakp$ because $\sU$ extends to the following two integral models over $k_{\frakp}$
		\[
		\begin{split}
			a_1\(u_1 + \omega u_2\)\(u_1^2 - \omega u_1u_2 + \omega^2 u_2^2\) &= a_0 - a_3 u_3^3, \\
			a_1\(u_1 + \omega^2 u_2\)\(u_1^2 - \omega^2 u_1u_2 + \omega u_2^2\) &= a_0 - a_3 u_3^3.
		\end{split}
		\]
		This shows that both entries of $(1 + \sigma)(\sB)$ are units of $\sO_{\frakp}$ and proves our claim in view of Lemma~\ref{lem:invp} and \eqref{eqn:Hilb-units-uniformiser}.
	\end{proof}
	
	\begin{lemma}
		\label{lem:a_0/a_3 unit}
		If $p \equiv 2 \bmod 3$ and $a_0/a_3 \in \Zp^{*}$, then $\inv_p \sB' = 0$.
	\end{lemma}
	
	\begin{proof}
		The proof follows from the fact that $a_0/a_3$ is a unit of $\Zp$ while for $p \equiv 2 \bmod 3$ any unit of $\Zp$ is a cube. The claim then follows from Lemma~\ref{lem:cube}.
	\end{proof}
	
	\begin{lemma}
		\label{lem:inv3}
		Assume that $(a_0, a_1, a_3) \equiv (2, 8, 5) \bmod 9$. Then $\inv_3 \sB' = 2/3$.
	\end{lemma}
	
	\begin{proof}
		Let $\frakp$ be the unique prime ideal above 3 in the ring of integers of $\QQ_3(\omega)$, it is generated by $\lambda = 1 - \omega$. Our proof rests upon Lemma~\ref{lem:invp} and \eqref{eqn: hilbert symbol at 3 for two units}, which confirm the claim of Lemma~\ref{lem:inv3} provided that
		\[
		\begin{split}
			\frac{a_0}{a_3} \bmod \ZZ_3^{*3} &\text{ expands as } 1 + b_2\lambda^2 + \dots, \\
			\frac{u_1 + \omega u_2}{u_1 + u_2} \bmod \ZZ_3^{*3} &\text{ expands as } 1 + c_1\lambda + \dots
		\end{split}
		\]
		and $b_2c_1 \equiv 2 \bmod 3$.
		
		Our assumptions imply that $a_0/a_3 \equiv 4 \bmod 9$ and thus it has expansion $4 = 1 + 3 = 1 - \lambda^2 - \lambda^3 + \dots$ according to \eqref{eqn: equations for lambda at the prime 3}. Hence $b_2 = -1$.
		
		The conditions on $a_0,a_1, a_3$ force any point in $\sU(\ZZ_3)$ to obey either $u_1^3 \equiv u_2^3 \equiv 1 \bmod 9$ or $u_1^3 \equiv u_2^3 \equiv -1 \bmod 9$. If $u_1^3 \equiv u_2^3 \equiv 1 \bmod 9$ or equivalently $u_1 \equiv u_2 \equiv 1 \bmod 3$, we have $u_1 + \omega u_2 = 1 + \omega + 3(k + m\omega)$ for some $k, m \in \ZZ$. Multiplying both $u_1 + \omega u_2$ and $u_1 + u_2$ by $-1$, which is a cube of $\ZZ_3^{*3}$, shows that $-u_1 - \omega u_2 \equiv 1 - 3 +\lambda \equiv 1 + \lambda \mod \lambda^2$ and $-u_1 - u_2 \equiv 1 \bmod \lambda^2$ by \eqref{eqn: equations for lambda at the prime 3}. Finally, multiplying $(u_1 + \omega u_2)/(u_1 + u_2)$ by $(u_1 + u_2)^3$ now confirms that for any choice of $\ZZ_3$-point with $u_1 \equiv u_2 \equiv 1 \bmod 3$ the coefficients $b_2, c_1$ above satisfy $b_2 = -1$ and $c_1 = 1$ whose product is congruent to 2 mod 3. The same analysis without the need of multiplying $u_1 + \omega u_2$ and $u_1 + u_2$ by $-1$ yields identical conclusion if $u_1^3 \equiv u_2^3 \equiv -1 \bmod 9$ and hence the claim.
	\end{proof}
	
	\begin{lemma}
		\label{lem:inv2}
		Assume that $(a_0, a_1, a_3) = (b, a, -2b)$ and $2 \nmid ab$. Then $\inv_2 \sB' \in \{0, 2/3\}$.
	\end{lemma}
	
	\begin{proof}
		Once again we employ Lemma~\ref{lem:invp}, which states that $\inv_2 \sB' = \inv_{\frakp} \sB$, where $\frakp$ is the unique prime ideal above $2$ in $\sO_k$. Since $2 \nmid ab$, the reduction of $\sU \bmod 2$ becomes
		\[
		u_1^3 + u_2^3 \equiv 1 \bmod 2.
		\]
		It is clear that any $\ZZ_2$-point must obey $u_1 \equiv 0 \bmod 2$ or $u_2 \equiv 0 \bmod 2$ but $u_1, u_2$ do not vanish mod 2 simultaneously. If $u_2 \equiv 0 \bmod 2$ then $\inv_{\frakp} \sB = 0$ as the second entry of $\sB$ becomes 1 mod 2 and thus it is a cube of $\ZZ_2$. On the other hand, if $u_1 \equiv 0 \bmod 2$, then the second entry of $\sB$ is $\omega$ and hence $\inv_{\frakp} \sB = 2/3$ by \eqref{eqn:Hilbert-unit-uniformiser-p=2(3)} confirming our claim.
	\end{proof}
	
	\section{Lower bound}
	\label{sec:lower}
	Keep notation as in section~\ref{sec:3-coeff}. To prove the lower bound of Theorem~\ref{thm:N(B)} we shall establish an asymptotic formula for the number of members of the subfamily given by 
	\begin{equation}
		\label{eqn:U-lower-bound}
		U: \quad au_1^3 + au_2^3 - 2bu_3^3 = b,
	\end{equation}
	with positive coprime $a, b$ such that $a \equiv 17 \bmod 18$, $b \equiv 11 \bmod 18$ and if $p \mid ab$, then $p \equiv 5 \bmod 6$. Each member of this family has a Brauer--Manin obstruction to the integral Hasse principle. 
	
	\begin{remark}
		\label{rem:Q-point}
		The compactification $X$ of \eqref{eqn:U-lower-bound} has a rational point $(0:1:-1:0) \in X(\QQ)$. Therefore $X$ is unirational \cite[Thm.~1]{Kol02} and thus $U(\QQ) \neq \emptyset$.
	\end{remark}
	
	\subsection*{Local solubility}
	We claim that $\sU(\Zp) \neq \emptyset$ for all $p$. It is clear that $\sU(\RR) \neq \emptyset$. For $p \nmid ab$ such that $p \neq 3, 7$ this follows by setting $u_3 = 0$. Then the reduction of $\sU$ mod $p$ is an elliptic curve minus the divisor at $\infty$ which has a smooth $\Fp$-point by the Hasse-Weil bound and hence it has a $\Zp$-point by Hensel. If $p = 7$ either $a/b$ is a cube mod 7, in which case one may set $u_2 = u_3 = 0$, or $a/b \equiv \pm 2, \pm 3 \bmod 7$, where local solubility once more is easily verified by setting $u_3 = 1$ and looking at 
	\[
	\frac{a}{b}\(u_1^3 + u_2^3\) \equiv 3 \bmod 7.
	\]
	Finally, local solubility at $p = 3$ is implied by the congruence conditions on $a,b$ as setting $u_2 = -1$ and $u_3 = 0$ reduces the defining equation of $\sU$ mod 9 to $u_1^3 \equiv -1 \bmod 9$. A unit of $\ZZ_3$ is a cube if and only if it is congruent to $\pm 1$ mod 9 and hence the claim of local solubility at 3. If $p \mid ab$, then $p \equiv 2 \mod 3$ and thus any unit mod $p$ is a cube. As $(a, b) = 1$ and $2 \nmid ab$ this is sufficient to deduce local solubility at such primes. 
	
	\subsection*{Values of the local invariant map} 
	We will now show that each $\sU$ has a Brauer--Manin obstruction to the integral Hasse principle as the sum of local invariant maps of $\sB'$ is never 0. If $p = \infty$ the local invariant map of $\sB'$ vanishes by Lemma~\ref{lem:infinity}. The same holds to $p \nmid 2ab$ provided that $p \neq 3$ by Lemma~\ref{lem:p-not-dividing-a_0a_3}. It also holds to any $p \mid ab$. Indeed, $2 \nmid ab$ and thus $-2^{-1}$, which is the first entry of $\sB'$, is a unit modulo any prime $p \mid ab$. By assumption any such prime satisfies $p \equiv 2 \bmod 3$ and thus Lemma~\ref{lem:a_0/a_3 unit} is applicable. The local invariant map equals $2/3$ at $p = 3$ by Lemma~\ref{lem:inv3}. It remains to show that it does not equal $1/3$ at $p = 2$ which follows from Lemma~\ref{lem:inv2}. This confirms that each $\sU$ has a Brauer--Manin obstruction to the integral Hasse principle as the sum of local invariant maps of each adelic point is either 1/3 or 2/3.
	
	\subsection{Establishing the lower bound}
	To simplify what follows, let
	\[
	\rho(n) =
	\begin{cases}
		1 &\text{if } p \mid n \implies p \equiv 5 \bmod 6, \\
		0 &\text{otherwise}.
	\end{cases} 
	\]
	Let $S(B)$ denote the number of $a, b \le B$ as in \eqref{eqn:U-lower-bound}, that is
	\[
	S(B)
	= \sum_{\substack{a \le B \\ a \equiv 17 \bmod 18}} \rho(a) \quad
	\sum_{\substack{b \le B, \ (b, a) = 1 \\ b \equiv 11 \bmod 18}} \rho(b).
	\]
	We claim that there is a real constant $c > 0$, such that
	\begin{equation}
		\label{eqn:S(B)}
		S(B) = c \frac{B^2}{\log B} + O\(\frac{B^2}{(\log B)^{3/2}}\).
	\end{equation}
	
	The coprimality condition $(a, b) = 1$ can be encoded using its indicator function $\sum_{d \mid (a, b)} \mu(d)$. The orthogonality of Dirichlet's characters mod $18$ now shows that
	\[
	S(B) = 
	\sum_{\chi_1, \chi_2 \bmod 18} \frac{\chi_1(17) \chi_2(5)}{36}
	\sum_{d \le B} \mu(d) \rho(d)
	\prod_{i = 1}^2  \chi_i(d) T_{\chi_i}\(\frac{B}{d}\),
	\]
	where for a real $x \ge 1$ and a Dirichlet character $\chi$ mod 18 we have defined $T_{\chi}(x)$ by
	\[
	T_{\chi}(x) = \sum_{a \le x} \rho(a)\chi(a).
	\]
	
	We split the sum over $d$ into two separate sums, one for the range $1 \le d \le B^{1/2}$ and one for $B^{1/2} < d \le B$. If $B^{1/2} < d \le B$, we may apply the trivial bound $B/d + O(1)$ to $T_{\chi_i}(B/d)$ for $i = 1, 2$. Bounding trivially the remaining sums over $d$ then gives
	\[
	\begin{split}
		S(B) 
		= \sum_{\chi_1, \chi_2 \bmod 18} \frac{\chi_1(17) \chi_2(5)}{36}
		\sum_{d \le B^{1/2}} \mu(d) \rho(d)
		\prod_{i = 1}^2  \chi_i(d) T_{\chi_i}\(\frac{B}{d}\)
		+ O\(B^{3/2}\).
	\end{split}
	\]
	
	Let $\chi_0$ be the trivial Dirichlet character mod 18. We claim that there is a real constant $c' > 0$, such that
	\begin{equation}
		\label{eqn:T_chi}
		T_{\chi}(x) =
		\begin{cases}
			c' x(\log x)^{-1/2} + O\(x(\log x)^{-1}\) &\text{if } \chi = \chi_0, \\
			O\(x(\log x)^{-1}\)  &\text{if } \chi \neq \chi_0.
		\end{cases}
	\end{equation}
	Combining this with $(\log(B/d))^{-1} = (\log B)^{-1}\(1 + O\(\log d/\log B\)\)$ now implies that
	\[
	\begin{split}
		S(B)
		&= c' \frac{B^2}{\log B} \sum_{d \le B^{1/2}} \frac{\rho(d)\mu(d)}{d^2} + O\(\frac{B^2}{(\log B)^{3/2}}\) \\
		&= c \frac{B^2}{\log B} + O\(\frac{B^2}{(\log B)^{3/2}}\),
	\end{split}
	\]
	where we have extended the range of summation of $d$ to $\infty$, which introduces a non-zero convergent sum (e.g. $1/\zeta(2) \neq 0$) that contributes to the constant in the main term, while its tail only produces a negligible error term. 
	
	It remains to verify \eqref{eqn:T_chi} in order for \eqref{eqn:S(B)} to hold. This can be done with the Landau-Selberg-Delange method. Consider the Dirichlet series $F(s, \chi)$ of this sum. We employ the standard notation $s = \sigma + it$. If $\sigma > 1$, the function $F(s, \chi)$ can be written as an Euler product
	\[
	F(s, \chi)
	= \sum_{a = 1}^{\infty} \frac{\rho(a)\chi(a)}{a^s}
	= \prod_{p \equiv 5 \bmod 6} \(1 - \frac{\chi(p)}{p^s}\)^{-1}.
	\]
	Let $\psi(\cdot) = \chi(\cdot) \(\frac{\cdot}{3}\)$. The binomial series expansion shows that
	\[
	\begin{split}
		F(s, \chi)
		&= \prod_{p} \(1 - \frac{1}{2}\(1 -\(\frac{p}{3}\)\)\frac{\chi(p)}{p^s}\)^{-1} \\
		&= \prod_p\(1 - \frac{\chi(p)}{p^s}\)^{-1/2}\(1 - \frac{\psi(p)}{p^s}\)^{1/2} E_p(s) 
		= \frac{L(s, \chi)^{1/2}}{L(s, \psi)^{1/2}}E(s).
	\end{split}
	\]
	Here $E_p(s) = 1 + O(p^{-2\sigma})$ and $E(s) = \prod_p E_p(s)$. 
	
	Note that $|\rho(a)\chi(a)| \le \rho(a)$ and $L(s, \chi_0) = \zeta(s)(1 - 1/2^s)(1 - 1/3^s)$. It is now clear that $F(s, \chi)$ satisfies the hypothesis of \cite[Thm.~II.5.2]{Ten15} with $N = 0$, $w = 1/2$, and with $z = 1/2$ if $\chi = \chi_0$, $z = -1/2$ if $\chi = \chi_0 \(\frac{\cdot}{3}\)$ and $z = 0$, for the remaining characters mod 18. Indeed, this is verified by \cite[Thm.~11.3, p.~360]{MV07} and \cite[Thm.~11.4, p.~362]{MV07} as $L(s, \chi_0 \(\frac{\cdot}{3}\))$ and $L(s, \chi')$ have no Siegel zeroes (eg. LMFDB), where $\chi'$ is the extension mod 6 of the quadratic Dirichlet character $\(\frac{\cdot}{3}\)$ mod 3. This confirms \eqref{eqn:T_chi} and completes the proof of the lower bound in Theorem~\ref{thm:N(B)}.
	\qed
	
	\section{Upper bounds}
	\label{sec:upper}
	We begin with an estimate of the number of $U$ defined in \eqref{eqn:U-main} with a non-trivial transcendental Brauer group. Let 
	\[
	\begin{split}
		M_{a_0}^{\tr}(B)
		&= \# \left\{ (a_1, a_2, a_3) \in [-B, B]^3 \cap \ZZ_{\prim}^3 \ : \ \Br U / \Br_1 U \text{ non-trivial} \right\}, \\
		M^{\tr}(B)
		&= \# \left\{ (a_0, a_1, a_2, a_3) \in [-B, B]^4 \cap \ZZ_{\prim}^4 \ : \
		\begin{aligned}
			&(a_1, a_2, a_3) = 1, \\ 
			&\Br U / \Br_1 U \text{ non-trivial}
		\end{aligned}
		\right\}.
	\end{split}
	\]
	
	\begin{proposition}
		\label{prop:trans}
		The following hold 
		\[
		M^{\tr}_{a_0}(B) \ll B(\log B)^6 \quad \text{and} \quad
		M^{\tr}(B) \ll B^2(\log B)^6.
		\]
	\end{proposition} 
	
	\begin{proof}
		It follows from Theorem~\ref{thm: Brauer group thm} that $U$ has a non-trivial transcendental Brauer group if and only if $a_1a_2a_3 \equiv 2 \bmod \QQ^{*3}$. It suffices to prove the claim for $M^{\tr}_{a_0}(B)$, as 
		\[
		M^{\tr}(B) 
		= M^{\tr}_{a_0}(B) \sum_{a_0 \le B} 1
		= BM^{\tr}_{a_0}(B) + O\(M^{\tr}_{a_0}(B)\).
		\]
		
		Note that the number of $a_1, a_2, a_3 \le B$ whose product $a_1a_2a_3 \in 2\QQ^{*3}$ is of the same magnitude as the number of those $a_1, a_2, a_3 \le B$ whose product is a cube. This follows, for example, from the fact that in the former problem at least one of $a_1, a_2, a_3$ has to be even. Thus any solution to $a_1a_2a_3 = 2n^3$ must come from a solution of $a_1a_2a_3 = n^3$ by multiplying one of the latter $a_i$ by 2. But each solution of $a_1a_2a_3 = n^3$ produces only finitely many solutions to $a_1a_2a_3 = 2n^3$ via doubling a coordinate and a permutation.
		
		The signs of the $a_i$ are immaterial for the order of magnitude of $M^{\tr}_{a_0}(B)$, as they only change the leading constant in the asymptotic. It thus suffices to count $0 < a_1, a_2, a_3 \le B$ with $\gcd(a_1, a_2, a_3) = 1$ and $a_1a_2a_3 \in \QQ^{*3}$. Thus 
		\[
		M^{\tr}_{a_0}(B) 
		\ll \sum_{\substack{a_1, a_2, a_3 \le B \\ (a_1, a_2, a_3) = 1, \ a_1a_2a_3 \in \QQ^{*3}}} 1
		\ll \sum_{\substack{a_1, a_2, a_3 \le B \\ (a_1, a_2, a_3) = 1}} 
		\sum_{\substack{n \le B \\ a_1a_2a_3 = n^3}}1.
		\] 
		The last quadruple sum has been investigated in the main result of \cite[p.~1]{HBM99}. It is asymptotically a constant times $B (\log B)^6$, thus proving our claim.
	\end{proof}
	
	We will also need the following two simple lemmas that will be applied at several instances in the study of the upper bounds considered here. 
	
\newlength{\mylength}%
\settowidth{\mylength}{$4(XYZ)^{1/2} + \ $}%

	\begin{lemma}
		\label{lem:triple-sum}
		Let $X \ge Y \ge Z \ge 1$ be real numbers. As $X, Y, Z \to \infty$, the following holds 
		\[
		\sum_{\substack{k \ell \le X, \ k m \le Y, \\ \ \ell m \le Z}} 1
		= \begin{cases}
			%4(XYZ)^{1/2} + O(X\log (YZ/X)) + O(X) + O(Y \log Z) &\text{if } X < YZ, \\
			%O(YZ) &\text{if } X \ge YZ.
			4(XYZ)^{1/2} + O(X\log (YZ/X) + X + Y \log Z) &\text{if } X < YZ, \\
			\makebox[\mylength][l]{}O(YZ) &\text{if } X \ge YZ.
		\end{cases}	
		\]
	\end{lemma}
	
	\begin{proof}
		Let $S$ denote the triple sum in the statement. Treating the sum over $m$ first gives
		\begin{equation}
			\label{eqn:S}
			S 
			= \sum_{k \ell \le X, \ k \le Y, \ \ell  \le Z} \(\min \left\{\frac{Y}{k},\frac{Z}{\ell}\right\} + O(1) \).
		\end{equation}
		Let $R$ be the sum corresponding to the error term above. Summing over one of $k, \ell$ first and then over the other gives
		\[
		R 
		=
		\begin{cases}
			X \log (YZ/X) + O(X) &\text{if } X < YZ, \\
			YZ + O(Y) &\text{if } X \ge YZ.
		\end{cases}
		\]
		
		We then consider separately the two contributions coming from the minimum, depending on if $Y/k < Z/\ell$ or vice versa. By doing so we get two identical sums $S_1$ and $S_2$ with arguments $1/k$ and $1/\ell$, respectively. We analyse $S_1$, the argument for $S_2$ being the same.
		\[
		\begin{split}
			S_1 
			&= Y\sum_{Y/Z < k \le Y} \frac{1}{k} \sum_{\ell \le \min \{kZ/Y, X/k\}} 1 
			= Y\sum_{Y/Z < k \le Y} \frac{1}{k}\(\min\left\{\frac{kZ}{Y}, \frac{X}{k} \right\} + O(1)\) \\
			&= Z\sum_{Y/Z < k \le \min\{Y, \sqrt{XY/Z\}}}1
			+ XY\sum_{\sqrt{XY/Z} < k \le Y}\frac{1}{k^2}
			+ O(Y \log Z),
		\end{split}
		\]
		as $\sum_{k \le x} 1/k = \log x + O(1)$. If $X \ge YZ$, the first sum is $YZ + O(Y)$, while the second sum is empty. On the other hand, if $X < YZ$, the asymptotic formula for the first sum on the last line is $(XYZ)^{1/2} + O(Y)$, while the second sum is convergent as $k \to \infty$. By completing the second sum we introduce an error term of size $O(X)$, while the completed sum equals $(XYZ)^{1/2} + O(Z)$. This altogether gives
		\[
		S_1 
		= \begin{cases}
			2(XYZ)^{1/2} + O(X) + O(Y \log Z) &\text{if } X < YZ, \\
			YZ + O(Y \log Z) &\text{if } X \ge YZ.
		\end{cases}
		\]
		
		If $Y/k \ge Z/\ell$ we may apply the same argument to $S_2$, yielding
		\[
		S_2 
		= \begin{cases}
			2(XYZ)^{1/2} + O(X) + O(Z \log Z) &\text{if } X < YZ, \\
			YZ + O(Y) + O(Z \log Z) &\text{if } X \ge YZ.
		\end{cases}
		\]
		This completes the proof, since $S = S_1 + S_2 + O(R)$ by \eqref{eqn:S}.
	\end{proof}
	
	The next two results are immediate corollaries of Lemma~\ref{lem:triple-sum}.
	
	\begin{corollary}
		\label{cor:triple-sum}
		Let $X, Y, Z \ge 1$ be real numbers. As $X, Y, Z \to \infty$, the following holds 
		\[
		\sum_{\substack{k \ell \le X, \ k m \le Y, \\ \ \ell m \le Z}} 1
		\ll (XYZ)^{1/2}.
		\]
	\end{corollary}
	
	\begin{lemma}
		\label{lem:sextupple-sum}
		Let $X, Y, Z, W \ge 1$ be real numbers tending to $\infty$. Then the following holds
		\[
		\begin{split}
			\sum_{\substack{k \ell m  \le X, \ kns \le Y \\ \ell nt \le Z, \ mst \le W}} 1 
			&\ll (XYZW)^{1/2}(\log (XYZW))^2.
		\end{split}
		\]
	\end{lemma}
	
	\begin{proof}
		Write the sextuple sum in the statement as
		\[
		\sum_{k \ell m  \le X} 
		\sum_{ns \le Y/k, \ nt \le Z/\ell, \ st \le W/m} 1.
		\]
		Lemma~\ref{lem:sextupple-sum} is now a straightforward application of Lemma~\ref{lem:triple-sum} and $\sum_{r \le X} \tau_3(r)/r^{1/2} \ll X^{1/2} (\log X)^2$, which follows for example from \cite[II.12]{SMC95} combined with partial summation, where $\tau_3(n)$ denotes the 3-divisor function.
	\end{proof}
	
	\subsection{Establishing the upper bounds}
	We are now in position to prove the upper bounds in this article.  Propositions~\ref{prop:p divides a_0} and \ref{prop:p divides a_1 but} show that, under the arithmetic conditions given there, the Brauer--Manin set obtained from algebraic Brauer elements is non-empty but at the same time it is strictly smaller than the integral adelic set $\sU(\mathbb{A}_{\mathbb{Z}})$. In view of Remark~\ref{rem:constant-inv}, the local invariant map at $\infty$ of any Brauer element is constant, as every real number is a cube. Thus the projection to the finite adeles $\prod_{p\neq \infty} \mathcal{U}(\mathbb{Z}_{p})$ preserves strict inclusions. Any surface satisfying the arithmetic conditions of Propositions~\ref{prop:p divides a_0} and \ref{prop:p divides a_1 but} then fails integral strong approximation off $\infty$ and at the same time has no Brauer--Manin obstruction to the integral Hasse principle coming from the algebraic Brauer group. 
	
	Therefore, in view of Proposition~\ref{prop:trans} it suffices to count surfaces failing the arithmetic conditions on $a_0, a_1, a_2, a_3$ given in Propositions~\ref{prop:p divides a_0} and \ref{prop:p divides a_1 but}. This will give upper bounds for the number of surfaces in the family which have a Brauer--Manin obstruction to the integral Hasse principle or which satisfy strong approximation off $\infty$. We proceed with establishing the upper bounds in Theorems~\ref{thm:N(B)}, \ref{thm:N-a_0(B)} and \ref{thm:N-a_123(B)} in this order.
	
	\begin{proof}[Proof of Theorem~\ref{thm:N(B)}]
		Propositions~\ref{prop:p divides a_0} and \ref{prop:p divides a_1 but} imply that if the local invariant map of $\sB$ does not surject at a given big enough prime $p$, then $p$ divides at least two of the coefficients of $\sU$. We can indicate that by introducing separate variables which keep track of the different $p$-adic multiplicities of pairs and of triples of coefficients of $\sU$. Then
		\begin{equation}
			\label{eqn:a_i-factorisation}
			a_i = \pm r_i b_i^3 
			\(\prod_{\substack{(j, k) \in S_{2, i} \\ (\ell, m) \in T_2}} u_{jk, \ell m}^{\delta(j, k, \ell, m)}\) 
			\(\prod_{\substack{(j, k) \in S_{3, i} \\ (\ell, m, n) \in T_3}} v_{0jk, \ell m n}^{\varepsilon(j, k, \ell, m, n)}\), 
			\quad i = 0, 1, 2, 3,
		\end{equation}
		where
		\[
		\begin{split}
			S_{2, i} &= \{(j, k) \ : \ 0 \le j < k \le 3 \text{ and } i \in \{j, k\}\}, \\ 
			S_{3, i} &= \{(j, k) \ : \ 1 \le j < k \le 3 \text{ and } i \in \{0, j, k\}\}, \\
			T_2 &= \{(\ell, m) \ : \ 1 \le \ell, m \le 3 \} \setminus \{(3, 3)\}, \\ 
			T_3 &= \{(\ell, m, n) \ : \ 1 \le \ell, m, n \le 3 \} \setminus \{(3, 3, 3)\},
		\end{split}
		\]
		\[
		\delta(j, k, \ell, m) =
		\begin{cases}
			\ell &\text{if } i = j,\\
			m &\text{if } i = k,
		\end{cases} 
		\quad {and} \quad
		\varepsilon(j, k, \ell, m, n) =
		\begin{cases}
			\ell &\text{if } i = 0,\\
			m &\text{if } i = j, \\
			n &\text{if } i = k.
		\end{cases}
		\]
		Moreover, the variables on the right hand side of \eqref{eqn:a_i-factorisation} are positive integers satisfying
		\begin{itemize}
			\item if $p \mid r_i$, then $p = 3$ for $i = 0$ and $p < 17$ for $i = 1, 2, 3$;
			\item $(a_i/r_i, r_i) = 1$ for $i = 0, 1, 2, 3$; 
			\item all $u_{jk, \ell m}$ and $v_{0jk, \ell m n}$ are squarefree. 
		\end{itemize}
		
		The definition of $r_i$ implies that for any constant $\alpha > 0$, we must have
		\begin{equation}
			\label{eqn:sum-r-i}
			\begin{split}
				\sum_{r_i \le X} \frac{1}{r_i^\alpha}
				\ll 1, \quad i=0,1,2,3.
			\end{split}
		\end{equation}
		Indeed, let $\varepsilon(n)$ be the indicator function of the set $\{n \in \ZZ_{\ge 0} \ : \ p \mid n \implies p < 17\}$. This function is non-negative and multiplicative. Hence 
		\[
		\sum_{n \le X} \frac{\varepsilon(n)}{n^\alpha}
		\ll \sum_{n = 1}^{\infty} \frac{\varepsilon(n)}{n^\alpha}
		= \prod_{p < 17}\(1 - \frac{1}{p^\alpha}\)^{-1}
		\ll 1.
		\]
		
		The signs of $a_i$ are immaterial to the count. Assuming that all of them are positive only changes $N(B)$ or $N^{'}(B)$ by a constant. Thus
		\begin{equation}
			\label{eqn:N(B)-sumstar}
			N(B) 
			\ll M^{\tr}(B) + \sumstar_{a_0, a_1, a_2, a_3 \le B} 1
			\quad \text{and} \quad
			N^{'}(B)
			\ll \sumstar_{a_0, a_1, a_2, a_3 \le B} 1.
		\end{equation}
		Here the superscript $*$ means that the $a_i$ counted in the respective sum satisfy \eqref{eqn:a_i-factorisation}. Note that we have removed the condition $(a_1, a_2, a_3) = 1$. Since we are summing non-negative integers, by doing so we potentially only enlarge the sum. We are now in position to apply Lemma~\ref{lem:sextupple-sum} to the sextuple sum over $u_{ij, 11}$, which after bounding logarithms trivially gives
		\begin{equation}
			\label{eqn:sumstar}
			\begin{split}
				\sumstar_{\substack{a_i \le B \\ i = 0, 1, 2, 3}} 1
				\ll B^2 (\log B)^2 
				\sum_{\substack{r_i, b_i, u_{jk, \ell m}, \\ v_{0jk, \ell m n}}} r_i^{-\frac{1}{2}} b_i^{-\frac{3}{2}} 
				\(\prod_{\substack{(j, k) \in S_{2, i} \\ (\ell, m) \in T_2'}} u_{jk, \ell m}^{-\frac{\ell + m}{2}} \)
				\(\prod_{\substack{(j, k) \in S_{3, i} \\ (\ell, m, n) \in T_3}} v_{0jk, \ell m n}^{-\frac{\ell + m + n}{2}} \),
			\end{split}
		\end{equation}
		where $T_2' = T_2 \setminus \{(1,1)\}$ and the summation is taken inside the region
		\[
		r_i b_i^3 
		\(\prod_{\substack{(j, k) \in S_{2, i} \\ (\ell, m) \in T_2'}} u_{jk, \ell m}^{\delta(j, k, \ell, m)}\) 
		\(\prod_{\substack{(j, k) \in S_{3, i} \\ (\ell, m, n) \in T_3}} v_{0jk, \ell m n}^{\varepsilon(j, k, \ell, m, n)}\) \le B, 
		\quad i= 0, 1, 2, 3.
		\]
		
		As we are summing non-negative integers, if we expand the summation range to
		\[
		r_i \le B, \quad 
		b_i \le B^{1/3}, \quad 
		u_{jk, \ell m} \le B^{1/\max\{\ell, m\}}, \quad
		v_{0jk, \ell m n} \le B^{1/\max\{\ell, m, n\}}
		\]
		for all possible $i, j, k, \ell, m, n$,  we are only enlarging the sum on the right hand side of \eqref{eqn:sumstar}. However, each of the sums over $r_i, b_i,u_{jk, \ell m}, v_{0jk, \ell m n}$ produced by this can be evaluated separately, and each of them is convergent if we let its summation variable go to infinity. This is confirmed by \eqref{eqn:sum-r-i} and the bound $\sum_{n > X} n^{-\alpha} = X^{- \alpha + 1}/(\alpha - 1) + O(X^{-\alpha})$, that holds whenever $\alpha > 1$. Moreover, by completing these sums we only introduce error terms with power saving and therefore 
		\[
		\sumstar_{a_0, a_1, a_2, a_3 \le B} 1
		\ll B^2 (\log B)^2 .
		\]
		This together with \eqref{eqn:N(B)-sumstar} and Proposition~\ref{prop:trans} completes the proof of Theorem~\ref{thm:N(B)}.
	\end{proof}
	
	\begin{proof}[Proof of Theorem~\ref{thm:N-a_0(B)}]
		The proof is similar to the one of the upper bound in Theorem~\ref{thm:N(B)}. If $\sU$ is counted in $N_{a_0}(B)$ or in $N_{a_0}^{'}(B)$, then Propositions~\ref{prop:p divides a_0} and \ref{prop:p divides a_1 but} imply that 
		\[
		a_i = \pm r_i b_i^3 
		\prod_{\substack{(j, k) \in S_{2, i} \\ (\ell, m) \in T_2}} u_{jk, \ell m}^{\delta(j, k, \ell, m)}, 
		\quad i = 1, 2, 3,
		\]
		where
		\[
		\begin{split}
			S_{2, i} &= \{(j, k) \ : \ 1 \le j < k \le 3 \text{ and } i \in \{j, k\}\}, \\ 
			T_2 &= \{(\ell, m) \ : \ 1 \le \ell, m \le 3 \} \setminus \{(3, 3)\},
		\end{split}
		\]
		\[
		\delta(j, k, \ell, m) =
		\begin{cases}
			\ell &\text{if } i = j,\\
			m &\text{if } i = k,
		\end{cases} 
		\]
		and the variables on the right hand side are positive integers satisfying
		\begin{itemize}
			\item if $p \mid r_i$, then $p \mid  a_0$ or $p < 17$ for $i = 1, 2, 3$;
			\item $(a_i/r_i, r_i) = 1$ for $i = 1, 2, 3$; 
			\item all $u_{jk, \ell m}$ are squarefree. 
		\end{itemize}
		
		Once again recalling Proposition~\ref{prop:trans}, summing over $u_{12, 11}, u_{13, 11}, u_{23, 11}$ in view of Corollary~\ref{cor:triple-sum}, and ignoring signs and the coprimality condition gives 
		\[
		N_{a_0}(B), \ N_{a_0}^{'}(B)
		\ll B (\log B)^6
		+ B^{3/2} \sum_{r_i, b_i, u_{jk, \ell m}} r_i^{-\frac{1}{2}} b_i^{-\frac{3}{2}}
		\prod_{\substack{(j, k) \in S_{2, i} \\ (\ell, m) \in T_2'}} u_{jk, \ell m}^{-\frac{\ell + m}{2}},
		\]
		where $T_2' = T_2 \setminus \{(1,1)\}$ and the summation is taken inside the region
		\[
		r_i b_i^3 \prod_{\substack{(j, k) \in S_{2, i} \\ (\ell, m) \in T_2'}} u_{jk, \ell m}^{\delta(j, k, \ell, m)} \le B, \quad i= 1, 2, 3.
		\]
		Replacing the summation region by boxes as in the proof of Theorem~\ref{thm:N(B)} then gives
		\[
		N_{a_0}(B), \ N_{a_0}^{'}(B) 
		\ll_{a_0} B^{3/2},
		\]
		which verifies the upper bounds claimed in Theorem~\ref{thm:N-a_0(B)}.
	\end{proof}
	
	\begin{proof}[Proof of Theorem~\ref{thm:N-a_123(B)}]
		As $a_1a_2a_3 \nequiv 2 \bmod \QQ^{*3}$, the transcendental part of $\Br U$ is trivial by Theorem~\ref{thm: Brauer group thm}. Similarly to the other upper bounds, $N_{a_1, a_2, a_3}(B)$ and $N_{a_1, a_2, a_3}^{}(B)$ then only count those $\sU$ for which
		\[
		a_0 = \pm ub^3,
		\]
		where if $p \mid u$, then $p \mid 3a_1a_2a_3$. Summing over $b$ first and then over $u$, while taking into account the definition of $u$ and the analogue of \eqref{eqn:sum-r-i} for it, gives
		\[
		N_{a_1, a_2, a_3}(B), \ N_{a_1, a_2, a_3}^{'}(B)
		\ll B^{1/3}\sum_{u \le B} \frac{\mathbf{1}_{p \mid u \implies p \mid 3a_1a_2a_3}}{u^{1/3}}
		\ll_{a_1 a_2 a_3} B^{1/3}.
		\]
		This completes the proof of Theorem~\ref{thm:N-a_123(B)}.
	\end{proof}
	
	\section{Examples of Brauer--Manin obstructions}\label{sec: examples of IBMO}
	
	The family of all affine diagonal cubic surfaces does not have a uniform generator of the Brauer group \cite[Thm.~2]{U14}. Getting a sharp lower bound on how many surfaces in this family have a Brauer--Manin obstruction thus requires a different approach than the one in section~\ref{sec: computing inv}. To get around the issue of lacking a uniform generator, we showcase the results from section~\ref{sec: computing inv} by giving instances of Brauer--Manin obstructions to the integral Hasse principle and to strong approximation without the need of having explicit representatives of Brauer elements. However, applying the method outlined here to the counting results is challenging as it requires controlling  uniformly the parameters $\epsilon$ and $\eta$, introduced in section~\ref{sec: computing inv}.
	
	\subsection{An interesting family}
	
	\begin{definition}
		We define $\mathcal X_{\ell,p,q} \subseteq \PP^3_{\mathbb Z}$ by
		\[
		x_{0}^3+\ell x_1^3 + pqx_{2}^3+q\ell^2x_{3}^3=0,
		\]
		where
		\begin{itemize}
			\item $\ell$, $p$ and $q$ are distinct primes,
			\item $q \equiv 8 \mod 9$,
			\item $\{p+9\ZZ,\ell+9\ZZ\} = \{2 + 9\ZZ, 5+9\ZZ\}$.
		\end{itemize}
		We define $\mathcal U_{\ell,p,q} \subseteq \PP^3_{\mathbb Z}$ as the complement of the curve $x_1=0$. We will write $S=\{3,\ell,p, q\}$ for the primes of bad reduction on $\mathcal{X}_{\ell,p,q}$.
	\end{definition}
	
	In accordance with the notation of section \ref{subsec: invariant map for generic families}, we will consistently write $\lambda = \ell$, $\mu = pq$ and $\nu = \ell/p$. Furthermore, recall that $\beta^3 = \lambda\nu = \ell^2/p$.
	
	Note that $\mathcal{X}_{\ell,p,q}$ is $k_v$-rational for $v \in \{3,\ell, p\}$, or equivalently, there is a cross ratio of the four coefficients which is a cube in $k_v$. Indeed, $\ell \cdot q \ell^2/pq \in \mathbb Z_\ell$ and $\ell \cdot pq/q\ell^2 \in \mathbb Z_q$ are cubes, since their valuations are multiples of $3$ and $\ell \equiv q \equiv 2 \mod 3$. Rationality at $3$ follows from Remark~\ref{rem:constant-inv} as $(q\ell^2 \cdot pq)/\ell \equiv p\ell \equiv 1 \mod 9$.
	
	The following two results show that the surfaces $\mathcal U_{\ell,p,q}$ are everywhere locally soluble, there is no obstruction to the integral Hasse principal for $\mathcal U_{\ell,p,q}$, but they always have an obstruction to integral strong approximation off $\infty$. Moreover, $\mathcal{U}_{\ell,p,q}$ should have a rational point (see Remark \ref{remark: CT conj implies that Ulpq has a rational point}), but there may be an integral Brauer--Manin obstruction to the integral Hasse principle on a different choice of integral model for $U_{\ell,p,q}$.
	
	\begin{lemma}\label{lem:loc sol for lpq}
		The affine cubic $\mathcal U_{\ell,p,q}$ is everywhere integrally locally soluble.
	\end{lemma}
	
	\begin{proof}
		Let $\mathcal{X}:=\mathcal{X}_{\ell,p,q}$ and $\mathcal{U}:= \mathcal{X}\setminus \mathcal{C}$ where $\mathcal{C}$ is the curve given by $x_{1}=0$ on $\mathcal{X}$. Local solubility at the archimedean place is immediate. For $v \not \in \{3,\ell,p,q\}$ $\mathcal X$, $\mathcal U$ and $\mathcal C$ have good reduction. The Hasse--Weil bound for genus $1$ curves and Weil's Theorem \cite[Thm.~27.1]{Man86}, for smooth cubic surfaces gives
		\[
		\#\mathcal X(\mathbb F_v) \geq q_v^2-2q_v+1 \quad \text{ and } \quad \#\mathcal C(\mathbb F_v) \leq q_v+1 + 2\sqrt{q_v},
		\]
		where $q_v := \#\mathbb F_v$. We find $\mathcal U(\mathbb F_v) >0$ for $q_v \geq 5$ and by Hensel's lemma we can lift such an $\mathbb F_v$-point to an $\mathbb Z_v$-point. If $q_v \equiv 2 \mod 3$ we have that any unit is a cube. This implies for $v \ne \ell$ we have $(-\sqrt[3]\ell,1,0,0) \in \mathcal U(\mathbb Z_v)$. For $v = \ell$ it shows we can lift $(\sqrt[3]{pq},0,-1,0) \in \mathcal U(\mathbb F_\ell)$ to a $\mathbb Z_\ell$-point.
		
		We are left with $v=3$. We have either $p\equiv 2 \mod 9$ and $\ell \equiv 5 \mod 9$, or $p\equiv 5 \mod 9$ and $\ell \equiv 2 \mod 9$, in both cases we can lift the point $(1,-1,-1,1) \in \mathcal U(\mathbb Z/9\mathbb Z)$ to a $\mathbb Z_3$-point.
	\end{proof}
	
	\begin{proposition}\label{prop:invariant maps for general family}
		We have $\Br U_{\ell,p,q}/\Br \QQ \cong \ZZ/3\ZZ$. Moreover, for a generator $\mathcal{B}'$ as in section \ref{subsec: invariant map for generic families} the following holds
		\[
		\inv_{v} \mathcal{B}' \colon \begin{cases}
			U_{\ell,p,q}(\QQ_v) \to \frac13\ZZ/\ZZ \text{ is constant} & \text{if } v \ne p,\\
			\mathcal{U}_{\ell,p,q}(\ZZ_v) \to \frac13\ZZ/\ZZ \text{ is surjective} & \text{if } v = p.
		\end{cases}
		\]
	\end{proposition}
	
	\begin{proof}
		The first statement follows from Proposition~\ref{prop: algebaric Brauer group for affine is the same as projective}(2), as the transcendental Brauer group is trivial by Theorem~\ref{thm: Brauer group thm}.
		
		Using Proposition~\ref{prop:p divides a_0} we see that
		\[
		\inv_{v}\mathcal{B}'\colon \mathcal U_{\ell,p,q}(\mathbb Z_p) \to \frac13\mathbb Z/\mathbb Z
		\]
		is surjective. This proves the claim for $v=p$.
		
		For the remaining primes $v$ we have that $v \not \in S$ is a prime of good reduction or $\mathcal{X}_{\ell,p,q}$ is $k_{v}$-rational. In those cases the invariant map is constant, see for example \cite[Lem.~5]{CTKS87}.
	\end{proof}
	
	\begin{remark}\label{remark: CT conj implies that Ulpq has a rational point}
		For $X_{\ell,p,q}:=\mathcal{X}_{\ell,p,q}\times \mathbb{Q}$ we have $X_{\ell,p,q}(\mathbb A_\QQ)^{\textup{Br}} \ne \emptyset$ \cite[\S 5, Prop.~2]{CTKS87} and under a conjecture by Colliot-Th\'el\`ene \cite[p. 174]{CT03} we have $X_{\ell,p,q}(\QQ) \ne \emptyset$. As $X_{\ell,p,q}$ is a smooth cubic surface $X_{\ell,p,q}(\mathbb{Q})$ is dense in $X_{\ell,p,q}$, hence the affine surface $U_{\ell,p,q}:=\mathcal{U}_{\ell,p,q}\times \mathbb{Q}$ always has $U_{\ell,p,q}(\mathbb{Q})\neq \emptyset$. For example, specialising $(\ell,p,q)=(2,5,17)$ we can find the rational point $(-1/2,1,1/2,-1/2)\in U_{2,5,17}(\mathbb{Q})$. Moreover, we see in Proposition \ref{prop: family of integral Brauer--Manin obstructions} that a different choice of integral model for $U_{2,5,17}$ has an integral Brauer--Manin obstruction.
	\end{remark}
	
	The exact value of $\inv_v \mathcal{B}'$ for $v \ne p$ depends on the choice of normalisation of $\mathcal{B}'$. To simplify the computation of the Brauer--Manin obstruction we make the following convenient choice by taking an $\epsilon \in \QQ(\beta)$ satisfying \eqref{eq:norm equation for epsilon} where $\epsilon$ a priori only lies in $\QQ(\omega,\alpha,\gamma)$.
	
	\begin{lemma}\label{lem:epsilon in Qbeta}
		There exists an $\epsilon \in \QQ(\beta)$ such that $\Norm_{\QQ(\beta)/\QQ} \epsilon = \mu$.
	\end{lemma}
	
	\begin{proof}
		By Remark~\ref{rem:good epsilon} we need to show that $\mu \in k$ is a local norm at all place $w \in \Omega_k$. Let $w_3$, $w_\ell$, $w_p$, $w_q$ be the unique places of $k$ above the indicated rational primes. Case (i) of Remark~\ref{rem:good epsilon} deals with the places $w \not \in \{w_3,w_\ell, w_p,w_q\}$. For the places $w \in \{w_\ell, w_q\}$ we use Case (ii) of the same remark. For the place $w=w_p$ we note that $\ell,q \in k_w$ are cubes, and $\ell^2/p = \beta^3$ is a norm. Hence $\mu = pq$ is also a norm at $w$. By the reciprocity law it follows that $\mu$ is also a local norm at $w = w_3$.
	\end{proof}
	
	To obtain a counterexample to the Hasse principle, we would want $\inv_p \mathcal{B}' \colon \mathcal U_{\ell, p, q}(\mathbb Z_p)\rightarrow \tfrac{1}{3}\mathbb{Z}/\mathbb{Z}$ to no longer be surjective. We will show that it at least assumes a single value on large subset of these points.

	\begin{proposition}\label{prop:inv at p}
		For a point $P_{p} = [x_{0}:x_{1}:x_{2}:x_{3}]\in \mathcal X_{\ell,p,q}(\mathbb{Z}_{p})$ such that $x_{0}\equiv 0 \mod p$ and $x_{1}  \not\equiv 0 \mod p$, we have
		\[
		\inv_{p}\mathcal{B}'(P_{p})=
		\begin{cases} 2/3 & \text{if } p \equiv 2 \mod 9;\\
			1/3 & \text{if } p \equiv 5 \mod 9.
		\end{cases}
		\]
	\end{proposition}
	
	\begin{proof}
		Let $w_p$ be the unique prime of $k$ above $p$. As $\nu = \ell/p$ and $\lambda/\nu = p$ are not cubes in $\mathbb{Q}_{p}$ but $\lambda = \ell \in \mathbb{Q}_p^{\ast 3}$, we fall into the case $G^{\langle p \rangle}=\langle s\rangle$ from Table \ref{table:inv maps}. Hence, we can compute $\inv_{p}\mathcal{B}'(P_{p})=\inv_{w_p} \mathcal B(P_p)$ as
		\[
		\left(\frac{x_{0}+\alpha\omega x_{1}}{x_{0}+\alpha\omega^2 x_{1}},\ell/p\right)_{\omega,w_p}
		=  \left(\frac{\alpha\omega }{\alpha\omega^2 },1/p\right)_{\omega,w_p}=(\omega,p)_{\omega,w_p}=-\frac{p^2-1}{9} \in \frac13\ZZ/\ZZ. \qedhere
		\]
	\end{proof}
	
	This shows that $\sum_v \inv_v \mathcal B'$ is constant on $\mathcal U'(\mathbb A_\mathbb Z)$ where $\mathcal U'$ is given by
	\[
	\mathcal U' \colon\quad p^3u_1^3+pqu_2^3 + q\ell^2u_3^3 = \ell.
	\]
	To determine whether there is a Brauer--Manin obstruction, that is $\sum_v \inv_v \mathcal B' \neq 0$, we will need to determine the constant value of $\inv_v \mathcal B'$ for all $v$ in the set of bad primes for $X$ and $\mathcal B'$ which we define by
	\[
	S_{\mathcal B'}:=\{v\colon v\mid w \text{ and } w(\epsilon)\ne 0\} \cup S.
	\]
	By Remark \ref{rem:good epsilon} we have $\inv_v \mathcal B'$ is identically $0$ for all other primes.
	
	In the next section we will construct a family where we have $S_{\mathcal B'}=S$, and hence $\inv_v \mathcal B'$ is identically zero for $v \not \in \{3,\ell,p,q\}$.
	
	\subsection{A family of counterexamples to the Hasse principle}
	
	The places where $\epsilon$ is not integral might still show up in the Brauer--Manin obstruction. However, in some cases we can ensure that there are no such primes. Let us restrict to $\ell=2$ and $p=5$, although many other families could be considered. Hence, we now consider the cubic surface
	\[
	\mathcal X = \mathcal X_{2,5,q}\colon x_{0}^3+2x_{1}^3+5qx_{2}^3+4q x_{3}^3=0,
	\]
	where $q$ is a prime that is $q\equiv 8 \mod 9$. Using the notation from subsection~\ref{subsec: invariant map for generic families} we have $\lambda = 2, \mu= 5q, \nu = 2/5 \text{ and } \beta^3 = 4/5$. It follows from Lemma \ref{lem:loc sol for lpq} that $\mathcal X(\mathbb{A}_{\mathbb{Z}})\neq \emptyset$.
	
	\begin{lemma}\label{lemma: epsilon for 2,5,q}
		For the surface $\mathcal U_{2,5,q}$, there exists an $\epsilon \in \mathcal{O}_{\mathbb Q(\beta)}$ for which $\Norm_{\mathbb{Q}(\beta)/\mathbb Q}(\epsilon)=-\mu$. In particular, such an $\epsilon$ is a unit away from $5$ and $q$.
	\end{lemma}
	
	\begin{proof}
		Let us write $M_2=\mathbb{Q}(\beta)$. In $\mathcal{O}_{M_2}$ we can see that both $(5)$ and $(q)$ are divisible by prime ideals $\mathfrak p_5$ and $\mathfrak p_q$ of inertia degree $1$ as $q \equiv 2 \mod 3$. Since $\text{Cl}(M_2) = \{0\}$ we see that the ideal $\mathfrak p_5\mathfrak p_q$ is generated by an element $\epsilon' \in \mathcal O_{M_2}$. By definition of $\mathfrak p_5$ and $\mathfrak p_q$ we see that $
		\text{Norm}_{M_2/\mathbb Q}\left((\epsilon')\right)=(5q)$. Hence $\text{Norm}(\epsilon')=5qu$ for some unit $u \in \mathbb Z^{\ast}$. So hence either $\epsilon'$ or $-\epsilon'$ is the $\epsilon$ we are looking for.
	\end{proof}

	Using Lemma~\ref{lemma: epsilon for 2,5,q} we can easily compute the invariant maps at almost all places.

	\begin{proposition}\label{prop: example 2 places 2,3,q}
		We have that the invariant maps for $\mathcal B'$ on $\mathcal U_{2,5,q}$ at the places $v \ne p$ are constant and satisfy
		\[
		\inv_v \mathcal B' =
		\begin{cases}
			0 & \text{ if } v \ne 3, p;\\
			1/3 & \text{ if } v =3.
		\end{cases}
		\]
	\end{proposition}
	
	\begin{proof}
		We already saw in Proposition~\ref{prop:invariant maps for general family} that all invariant maps apart from $v=p$ are constant.
		
		For the primes of good reduction $v \not \in S$ we refer to \cite[p. 31]{CTKS87}. For the three remaining primes we compute the invariant map at a single point using Table \ref{table:inv maps}.
		
		\noindent \underline{$v =3$}\\
		Note that $\lambda,\nu, \lambda/\nu$ are not cubes in $\mathbb{Q}_{3}$ but $\lambda\nu = 4/5$ is a cube in $\mathbb{Q}_{3}$, thus we fall into the case $G^{\langle 3 \rangle}= \langle q \rangle$. Since we can take $\epsilon \in \mathbb Q(\beta)$ by Lemma~\ref{lem:epsilon in Qbeta} we have $\eta=1$ by Remark~\ref{rem:good epsilon}. Futhermore, as $r$ fixes $\epsilon$ we have for all $P_{3}\in X(\mathbb{Q}_{3})$
		\[
		\inv_{3}\mathcal B'(P_{3})=\inv_{w_3} \mathcal B(P_3) = (r(\epsilon)/\epsilon,2)_{\omega,w_3}-(h(P_{3}),2)_{\omega,w_3} = -(h(P_{3}),2)_{\omega,w_3},
		\]
		where $w_3$ is the unique place of $k$ dividing $3$. As $\inv_{3}$ is constant on $X(\mathbb{Q}_{3})$, it is sufficient to compute the invariant map at one point $P_{3}:=[-\sqrt[3]{q}:0:1:-1]\in X(\mathbb{Q}_{3})$. This gives
		\[
		(h(P_{3}),2)_{\omega,w_3}=\left(\frac{1-\omega\beta}{1-\beta},2\right)_{\omega, w_3}=2/3.
		\]
		
		\noindent \underline{$v =\ell =2$}\\
		Let $w_2$ be the one place of $k$ above $2$. Note that $\lambda$ and $\nu$ are not cubes in $\mathbb{Q}_{2}$ but $\lambda/\nu = 5$ is a cube in $\mathbb{Q}_{2}$, thus we fall into the case $G^{\langle 2 \rangle}=\langle r \rangle$. By Lemma \ref{lemma: epsilon for 2,5,q}, $\epsilon \in \mathbb{Q}(\beta)$, hence $q(\epsilon)=\epsilon$ and for all $P_{2}\in X(\mathbb{Q}_{2})$
		\[
		\inv_{2}\mathcal{B}'(P_{2})=(q(\epsilon)/\epsilon,2/5)_{\omega,w_2}=(1,2/5)_{\omega,w_2}=0.
		\]
		
		\noindent \underline{$v =q$}\\
		As $\nu$ is a cube in $\mathbb{Q}_{q}$ we fall into the case $G^{\langle q \rangle}=\langle t \rangle$, hence $\inv_{q}= 0$.
	\end{proof}
	
	\begin{proposition}\label{prop: example 1 prime 5}
		Let $P_{5} = [x_{0}:x_{1}:x_{2}:x_{3}]\in \mathcal X(\mathbb{Z}_{5})$ such that $x_{0}\equiv 0 \mod 5$ and $x_{1}  \not\equiv 0 \mod 5$, then
		\[
		\inv_{5}\mathcal B'(P_{5})=1/3.
		\]
	\end{proposition}
	\begin{proof}
		This is precisely Proposition~\ref{prop:inv at p} for the case $\ell=2$ and $p=5$.
	\end{proof}
	
	\begin{corollary}\label{cor: example 1 invariant map on cubic surface}
		Let $(P_{v})_{v\in \Omega_{\mathbb{Q}}}\in \mathcal{X}(\mathbb{A}_{\mathbb{Z}})$ such that $P_{5}=[x_{0}:x_{1}:x_{2}:x_{3}]$ where $x_{0}\equiv 0$ mod 5 and $x_{1}\not\equiv 0$ mod 5. Then \[
		\sum\limits_{v\in \Omega_{\mathbb{Q}}}\inv_{v}\mathcal{B}'(P_{v})=2/3.
		\]
	\end{corollary}
	\begin{proof}
		As $\mathcal{B}'$ is of order 3 in $\Br X$ we have $\inv_{\infty}\mathcal{B}'$ is constant as $\Br \mathbb{R}\cong \mathbb{Z}/2\mathbb{Z}$. By our choice of normalisation we have $\inv_{\infty}\mathcal{B}'(P_{\infty})=0$ for all $P_{\infty}\in X(\mathbb{R})$. Then by Propositions \ref{prop: example 2 places 2,3,q} and \ref{prop: example 1 prime 5} we have \[
		\sum\limits_{v\in \Omega_{\mathbb{Q}}}\inv_{v}\mathcal{B}'(P_{v})=2/3.\qedhere
		\]
	\end{proof}
	
	The above result shows that strong approximation on $\mathcal U_{2,5,q}$ fails. We can use this result to produce affine diagonal cubics for which the Brauer--Manin obstruction obstructs the integral Hasse principle.
	
	\begin{proposition}\label{prop: family of integral Brauer--Manin obstructions}
		Let $q \equiv 8 \mod 9$ be a prime number. Consider the surface $\mathcal U'$ given by
		\[
		5^3u_{1}^3+5qu_{2}^3+4qu_{3}^3=2 \subset \mathbb{A}^{3}_{\mathbb{Z}}.
		\]
		We have $\mathcal U'(\mathbb A_\ZZ) \ne \emptyset$, but $\mathcal U'(\mathbb A_\ZZ)^{\text{Br}} = \emptyset$. In particular, $\mathcal U'(\ZZ) = \emptyset$.
	\end{proposition}
	\begin{proof}
		Let \[
		i:\mathcal{U}'\rightarrow \mathcal{X}, (u_{1},u_{2},u_{3})\mapsto [5u_{1}:u_{2}:u_{3}:1]
		\] and $U':=\mathcal{U}'\times_{\mathbb{Z}}\mathbb{Q}$. Since $\mathcal U_{2,5,q}$ and $\mathcal U'$ are isomorphic over $\mathbb Z[1/5]$ the local solubility is immediate away from $5$. At the place $5$ we can take the point $(0, 0, 1/\sqrt[3]{2q}) \in \mathcal U'(\ZZ_5)$.

		Since the surfaces are isomorphic over $\QQ$ we have that $i^{*}\mathcal{B}'$ generates $\Br U'/\Br_0 U'$. By functoriality and Corollary \ref{cor: example 1 invariant map on cubic surface} we see that for all $(P_{v})_{v\in \Omega_{\mathbb{Q}}}\in \mathcal{U}'(\mathbb{A}_{\mathbb{Z}})$
		\[
		\sum\limits_{v\in \Omega_{\mathbb{Q}}}\inv_{v}i^{*}\mathcal{B}'(P_{v})=2/3.
		\] Hence, $\mathcal{U}'(\mathbb{A}_{\mathbb{Z}})^{\Br}=\emptyset$.
	\end{proof}

\bibliographystyle{amsalpha}{}
\bibliography{references}

\providecommand{\bysame}{\leavevmode\hbox to3em{\hrulefill}\thinspace}
\providecommand{\MR}{\relax\ifhmode\unskip\space\fi MR }
\providecommand{\MRhref}[2]{%
  \href{http://www.ams.org/mathscinet-getitem?mr=#1}{#2}
}
\providecommand{\href}[2]{#2}
\begin{thebibliography}{CTWX20}

\bibitem[BB14a]{BB14b}
R.~de~la Bret\`eche and T.~D. Browning, \emph{Contre-exemples au principe de
  {Hasse} pour certains tores coflasques}, Journal de th\'eorie des nombres de
  Bordeaux \textbf{26} (2014), no.~1, 25--44.

\bibitem[BB14b]{BB14}
\bysame, \emph{Density of {C}h\^{a}telet surfaces failing the {H}asse
  principle}, Proc. Lond. Math. Soc. (3) \textbf{108} (2014), no.~4,
  1030--1078.

\bibitem[BBL16]{BBL16}
M.~J. Bright, T.~D. Browning, and D.~Loughran, \emph{Failures of weak
  approximation in families}, Compos. Math. \textbf{152} (2016), no.~7,
  1435--1475.

\bibitem[Ber17]{Ber17}
J.~Berg, \emph{Obstructions to the integral {H}asse principle for affine
  {C}hâtelet surfaces.}, \texttt{arXiv:1710.07969} (2017).

\bibitem[BL19]{BL19}
M.~J. Bright and J.~Lyczak, \emph{A uniform bound on the {B}rauer groups of
  certain log {K}3 surfaces}, Michigan Math. J. \textbf{68} (2019), no.~2,
  377--384.

\bibitem[Boo19]{B19}
A.~R. Booker, \emph{Cracking the problem with 33}, Res. Number Theory
  \textbf{5} (2019), no.~3, Paper No. 26, 6.

\bibitem[BS21]{BS21}
A.~R. Booker and A.~V. Sutherland, \emph{On a question of {M}ordell}, Proc.
  Natl. Acad. Sci. USA \textbf{118} (2021), no.~11, Paper No. 2022377118.

\bibitem[BSD04]{BS16}
M.J. Bright and P.~Swinnerton-Dyer, \emph{Computing the {B}rauer--{M}anin
  obstructions}, Mathematical Proceedings of the Cambridge Philosophical
  Society, vol. 137, Cambridge University Press, 2004, pp.~1--16.

\bibitem[BST22]{BST22}
T.~D. Browning, E.~Sofos, and J.~Teräväinen, \emph{Bateman-{H}orn, polynomial
  {C}howla and the {H}asse principle with probability 1},
  \texttt{arXiv:2212.10373} (2022).

\bibitem[CLR21]{AG21}
G.~Chiloyan and \'{A}. Lozano-Robledo, \emph{A classification of
  isogeny-torsion graphs of {$\mathbb{ Q}$}-isogeny classes of elliptic
  curves}, Trans. London Math. Soc. \textbf{8} (2021), no.~1, 1--34.

\bibitem[CT03]{CT03}
J.-L. Colliot-Th{\'e}l{\`e}ne, \emph{Points rationnels sur les fibrations},
  pp.~171--221, Springer Berlin Heidelberg, 2003.

\bibitem[CTKS87]{CTKS87}
J.-L. Colliot-Th\'{e}l\`ene, D.~Kanevsky, and J.-J. Sansuc,
  \emph{Arithm\'{e}tique des surfaces cubiques diagonales}, Diophantine
  approximation and transcendence theory ({B}onn, 1985), Lecture Notes in
  Math., vol. 1290, Springer, Berlin, 1987, pp.~1--108.

\bibitem[CTS21]{CS21}
J.-L. Colliot-Th\'{e}l\`ene and A.~N. Skorobogatov, \emph{The
  {B}rauer-{G}rothendieck group}, Ergebnisse der Mathematik und ihrer
  Grenzgebiete, vol.~71, Springer, Cham, 2021.

\bibitem[CTW12]{CTW12}
J.-L. Colliot-Th{\'e}l{\`e}ne and O.~Wittenberg, \emph{Groupe de {B}rauer et
  points entiers de deux familles de surfaces cubiques affines}, Amer. J. Math.
  \textbf{134} (2012), no.~5, 1303--1327.

\bibitem[CTWX20]{CTWX20}
J.-L. Colliot-Th\'el\`ene, D.~Wei, and F.~Xu, \emph{{Brauer-Manin obstruction
  for Markoff surfaces}}, Ann. Sc. Norm. Super. Pisa Cl. Sci. \textbf{XXI}
  (2020), no.~5, 1257--1313.

\bibitem[CTX09]{CTX09}
J.-L. Colliot-Th{\'e}l{\`e}ne and F.~Xu, \emph{Brauer-{M}anin obstruction for
  integral points of homogeneous spaces and representation by integral
  quadratic forms}, Compos. Math. \textbf{145} (2009), no.~2, 309--363, With an
  appendix by D. Wei and F. Xu.

\bibitem[Dic66]{Dic66}
L.~E. Dickson, \emph{History of the theory of numbers. {V}ol. {II}:
  {D}iophantine analysis}, Chelsea Publishing Co., New York, 1966.

\bibitem[GLN22]{GLN22}
D.~Gvirtz, D.~Loughran, and M.~Nakahara, \emph{Quantitative arithmetic of
  diagonal degree 2 {K}3 surfaces}, Math. Ann. \textbf{384} (2022), 1--75.

\bibitem[GS22]{GS22}
A.~Ghosh and P.~Sarnak, \emph{Integral points on {M}arkoff type cubic
  surfaces}, Invent. Math. \textbf{229} (2022), no.~2, 689--749.

\bibitem[Har17]{H17}
Y.~Harpaz, \emph{Geometry and arithmetic of certain log {K}3 surfaces}, Ann.
  Inst. Fourier (Grenoble) \textbf{67} (2017), no.~5, 2167--2200.

\bibitem[HB92]{HB92}
D.~R. Heath-Brown, \emph{The density of zeros of forms for which weak
  approximation fails}, Math. Comp. \textbf{59} (1992), no.~200, 613--623.

\bibitem[HBM99]{HBM99}
D.~R. Heath-Brown and B.~Z. Moroz, \emph{The density of rational points on the
  cubic surface ${X_0^3 = X_1 X_2 X_3}$}, Math. Proc. Camb. Phil. Soc.
  \textbf{385} (1999), no.~3, 385 -- 395.

\bibitem[Jah14]{J14}
J.~Jahnel, \emph{Brauer groups, {T}amagawa measures, and rational points on
  algebraic varieties}, Mathematical Surveys and Monographs, vol. 198, American
  Mathematical Society, Providence, RI, 2014.

\bibitem[Jou83]{J83}
J.-P. Jouanolou, \emph{Th\'{e}or\`emes de {B}ertini et applications}, Progress
  in Mathematics, vol.~42, Birkh\"{a}user Boston, Inc., Boston, MA, 1983.

\bibitem[JS17]{JS17}
J.~Jahnel and D.~Schindler, \emph{Del {P}ezzo surfaces of degree four violating
  the {H}asse principle are {Z}ariski dense in the moduli scheme}, Ann. Inst.
  Fourier (Grenoble) \textbf{67} (2017), no.~4, 1783--1807.

\bibitem[Kol02]{Kol02}
J.~Koll\'{a}r, \emph{Unirationality of cubic hypersurfaces}, J. Inst. Math.
  Jussieu \textbf{1} (2002), no.~3, 467--476.

\bibitem[LM21]{LM20}
D.~Loughran and V.~Mitankin, \emph{Integral {H}asse principle and strong
  approximation for {M}arkoff surfaces}, IMRN (2021), no.~18, 14086--14122.

\bibitem[Lyc23]{L23}
J.~Lyczak, \emph{Order 5 {B}rauer-{M}anin obstructions to the integral {H}asse
  principle on log {K}3 surfaces}, Ann. Inst. Fourier (Grenoble) \textbf{73}
  (2023), no.~2, 447--478.

\bibitem[Man71]{Man71}
Y.~I. Manin, \emph{Le groupe de {B}rauer-{G}rothendieck en g\'{e}om\'{e}trie
  diophantienne}, Actes du {C}ongr\`es {I}nternational des {M}ath\'{e}maticiens
  ({N}ice, 1970), {T}ome 1, 1971, pp.~401--411.

\bibitem[Man86]{Man86}
\bysame, \emph{Cubic forms}, second ed., North-Holland Mathematical Library,
  vol.~4, North-Holland Publishing Co., Amsterdam, 1986.

\bibitem[Mil16]{MIL16}
J.S. Milne, \emph{{\'E}tale cohomology (pms-33), volume 33}, Princeton
  university press, 2016.

\bibitem[Mit17]{Mit17}
V.~Mitankin, \emph{Failures of the integral {H}asse principle for affine
  quadric surfaces}, J. Lond. Math. Soc. (2) \textbf{95} (2017), no.~3,
  1035--1052.

\bibitem[Mit20]{Mit20}
\bysame, \emph{Integral points on generalised affine {C}h\^{a}telet surfaces},
  Bull. Sci. Math. \textbf{159} (2020), 102830, 20.

\bibitem[MNS24]{MNS22}
V.~Mitankin, M.~Nakahara, and S.~Streeter, \emph{Semi-integral
  {B}rauer--{M}anin obstruction and quadric orbifold pairs}, Trans. Amer. Math.
  Soc. \textbf{377} (2024), 4435--4480.

\bibitem[Mor42]{Mor42}
L.J. Mordell, \emph{On sums of three cubes}, J. Lond. Math. Soc. \textbf{s1-17}
  (1942), no.~3, 139 -- 144.

\bibitem[MS22]{MS22}
V.~Mitankin and C.~Salgado, \emph{Rational points on {D}el {P}ezzo surfaces of
  degree four}, Int. J. Number Theory \textbf{18} (2022), no.~9, 2099--2127.

\bibitem[MV07]{MV07}
H.~L. Montgomery and R.~C. Vaughan, \emph{Multiplicative number theory. {I}.
  {C}lassical theory}, Cambridge Studies in Advanced Mathematics, vol.~97,
  Cambridge University Press, Cambridge, 2007.

\bibitem[Neu13]{Neu13}
J.~Neukirch, \emph{Class field theory}, Springer, Heidelberg, 2013, The Bonn
  lectures.

\bibitem[Rom19]{Rom19}
N.~Rome, \emph{A positive proportion of {H}asse principle failures in a family
  of {C}hâtelet surfaces}, Int. J. Number Theory \textbf{15(06)} (2019),
  1237--1249.

\bibitem[San22]{San22}
T.~Santens, \emph{Integral points on affine quadric surfaces}, Journal de
  Théorie des Nombres de Bordeaux \textbf{34} (2022), no.~1, 141--161.

\bibitem[San23]{San23}
\bysame, \emph{Diagonal quartic surfaces with a {B}rauer--{M}anin obstruction},
  Compositio Math. \textbf{159(4)} (2023), 259--710.

\bibitem[Ser79]{S79}
J.-P. Serre, \emph{Local fields}, Graduate Texts in Mathematics, vol.~67,
  Springer-Verlag, New York-Berlin, 1979, Translated from the French by Marvin
  Jay Greenberg.

\bibitem[SMC95]{SMC95}
J.~Sándor, D.~S. Mitrinovic, and B.~Crstici, \emph{Handbook of number theory i
  edition: 1}, Springer, 1995.

\bibitem[Ten15]{Ten15}
G.~Tenenbaum, \emph{Introduction to analytic and probabilistic number theory},
  third ed., Graduate Studies in Mathematics, vol. 163, American Mathematical
  Society, Providence, RI, 2015.

\bibitem[Uem14]{U14}
T.~Uematsu, \emph{On the {B}rauer group of diagonal cubic surfaces}, Q. J.
  Math. \textbf{65} (2014), no.~2, 677--701.

\bibitem[Upp24]{U23}
H.~Uppal, \emph{Integral points on symmetric affine cubic surfaces},
  Manuscripta Math. \textbf{173} (2024), no.~3-4, 1305--1331.

\bibitem[vdW91]{VDW91}
B.~L. van~der Waerden, \emph{Algebra. {V}ol. {I}}, german ed., Springer-Verlag,
  New York, 1991, Based in part on lectures by E. Artin and E. Noether.

\bibitem[Wan17]{W23}
V.~Wang, \emph{Sums of cubes and the ratios conjectures},
  \texttt{arXiv:2108.03398} (2017).

\end{thebibliography}
\end{document}